\documentclass[a4paper,12pt]{amsart}
\makeatletter
\@namedef{subjclassname@2020}{\textup{2020} Mathematics Subject Classification}
\makeatother

\usepackage[doi=false,isbn=false,sorting=nyt,backend=biber,bibencoding=utf8,maxnames=10,maxalphanames=5,giveninits=true, style=alphabetic]{biblatex}
\bibliography{master}

\PassOptionsToPackage{hyphens}{url}
\usepackage[colorlinks,plainpages,hypertexnames=false,plainpages=false]{hyperref}
\hypersetup{urlcolor=blue, citecolor=blue, linkcolor=blue}

\usepackage{amsmath}
\usepackage{amsthm}
\usepackage{amssymb}
\usepackage{amsfonts}
\usepackage{enumitem}
\usepackage{listings}
\usepackage{xspace} 
\usepackage[noend]{algorithmic} 

\usepackage{subcaption}

\usepackage{booktabs, multirow}
\usepackage{mathtools}

\usepackage{array}
\newcolumntype{L}[1]{>{\raggedright\let\newline\\\arraybackslash\hspace{0pt}}m{#1}}
\newcolumntype{C}[1]{>{\centering\let\newline\\\arraybackslash\hspace{0pt}}m{#1}}
\newcolumntype{R}[1]{>{\raggedleft\let\newline\\\arraybackslash\hspace{0pt}}m{#1}}
\usepackage{bm}

\usepackage{tikz}
\usetikzlibrary{arrows}
\usetikzlibrary{calc}
\usetikzlibrary{shapes}
\usetikzlibrary{graphs}

\usetikzlibrary{automata,petri}
\tikzset{%
  transition/.style={rectangle,minimum size=6mm,draw},
  place/.style={circle,minimum size=6mm,draw},
  database/.style={
    minimum width=2cm,minimum height=1cm,cylinder,
    shape border rotate=90,aspect=0.2,draw
  }
}

\newtheoremstyle{theoremstyle}
{10pt}      
{5pt}       
{\itshape}  
{}          
{\bfseries} 
{}         
{ }      
{}          

\newtheoremstyle{algorithmstyle}
{10pt}      
{5pt}       
{}  
{}          
{\bfseries} 
{}         
{ }      
{}          

\newtheoremstyle{examplestyle}
{10pt}      
{5pt}       
{}          
{}          
{\bfseries} 
{}         
{ }      
{}          

\theoremstyle{theoremstyle}
\newtheorem{theorem}{Theorem}[section]
\newtheorem{lemma}[theorem]{Lemma}
\newtheorem{proposition}[theorem]{Proposition}
\newtheorem{corollary}[theorem]{Corollary}
\theoremstyle{examplestyle}
\newtheorem{example}[theorem]{Example}
\newtheorem{definition}[theorem]{Definition}
\newtheorem{remark}[theorem]{Remark}
\newtheorem{convention}[theorem]{Convention}
\newtheorem{conjecture}[theorem]{Conjecture}

\theoremstyle{algorithmstyle}
\newtheorem{algorithm}[theorem]{Algorithm}

\newcommand{\NN}{\mathbb{N}}
\newcommand{\CC}{\mathbb{C}}
\newcommand{\CCt}{{\mathbb C\{\!\{t\}\!\}}}
\newcommand{\RR}{\mathbb{R}}
\newcommand{\QQ}{\mathbb{Q}}
\newcommand{\ZZ}{\mathbb{Z}}
\newcommand{\FF}{\mathbb{F}}

\newcommand{\suchthat}{\;\ifnum\currentgrouptype=16 \middle\fi|\;}
\newcommand{\bigmid}{\left.\vphantom{\big\{} \suchthat \vphantom{\big\}}\right.}

\newcommand\SmallSetOf[2]{\{{#1}\,|\,{#2}\}}

\newcommand{\Gfan}{\textsc{Gfan}\xspace}
\newcommand{\Polymake}{\textsc{polymake}\xspace}
\newcommand{\Singular}{\textsc{Singular}\xspace}
\newcommand{\GPISpace}{\textsc{GPI-Space}\xspace}

\DeclareMathOperator{\Dress}{Dr}
\DeclareMathOperator{\Grass}{Gr}
\DeclareMathOperator{\TropGrass}{TGr}
\DeclareMathOperator{\initial}{in}
\DeclareMathOperator{\Trop}{Trop}
\DeclareMathOperator{\Spec}{Spec}
\DeclareMathOperator{\lcm}{lcm}
\DeclareMathOperator{\LM}{LM}
\DeclareMathOperator{\sgn}{sgn}
\DeclareMathOperator{\spoly}{spoly}

\begin{document}

\title[Parallel computation of tropical varieties]{Parallel Computation of tropical varieties, their positive part, and tropical Grassmannians}
\subjclass[2020]{14T15, 68W10, 68W30, 14Q15, 14M15, 52B15}

\author{Dominik Bendle}
\address{TU Kaiserslautern}
\email{bendle@rhrk.uni-kl.de}

\author{Janko B\"ohm}
\address{TU Kaiserslautern}
\email{boehm@mathematik.uni-kl.de}

\author{Yue Ren}
\address{Swansea University}
\email{yue.ren@swansea.ac.uk}

\author{Benjamin Schr\"oter}
\address{Binghamton University}
\email{schroeter@math.binghamton.edu}

\thanks{Research by Janko B\"ohm on this work has been supported by Project II.5 of SFB-TRR 195: ``Symbolic Tools in Mathematics and their Application'' of Deutsche Forschungsgemeinschaft. Research by Yue Ren and Benjamin Schr\"oter was partially supported by the Institute Mittag-Leffler during the during the program ``Tropical Geometry, Amoebas and Polytopes''. The authors would like to thank the institute for its hospitality.}

\date{\today}

\maketitle

\begin{abstract}
  In this article, we present a massively parallel framework for computing tropicalizations of algebraic varieties which can make use of finite symmetries. We compute the tropical Grassmannian $\TropGrass_0(3,8)$, and show that it refines the $15$-dimensional skeleton of the Dressian $\Dress(3,8)$ with the exception of $23$ special cones for which we construct explicit obstructions to the realizability of their tropical linear spaces. Moreover, we propose algorithms for identifying maximal-dimensional tropical cones which belong to the positive tropicalization. These algorithms exploit symmetries of the tropical variety even though the positive tropicalization need not be symmetric. We compute the maximal-dimensional cones of the positive Grassmannian $\TropGrass^+(3,8)$ and compare them to the cluster complex of the classical Grassmannian $\Grass(3,8)$.
\end{abstract}

\section{Introduction}
\noindent
Tropical geometry studies combinatorial objects arising from systems of polynomial equations. These so-called \emph{tropical varieties} arise naturally in many areas within and beyond mathematics, such as algebraic geometry \cite{Mikhalkin05}, combinatorics \cite{ArdilaKlivans:2006}, as well as optimization \cite{ABGJ18}, biology \cite{SpeyerSturmfels:2004,YZZ17}, economics \cite{TY15,BK19}, and physics \cite{HM06,HJ11}. Wherever they emerge, tropical varieties are often tied to concrete computational problems, which is why their computation is of key importance for many applications. 

The process of tropicalization associates to any algebraic variety a tropical variety. In this paper, we present an implementation for computing tropicalizations in parallel. It builds on the approach originally developed by Bogart, Jensen, Speyer, Sturmfels, and Thomas \cite{BJSST07} and implemented by Jensen in \Gfan \cite{gfan}. It is also available in the computer algebra system \textsc{Singular} \cite{singular}, see \cite{MR19}. The method is based on the fact that the tropicalization of an irreducible algebraic variety is the support of a polyhedral cell complex, which is connected in codimension 1. The tropical variety is determined via a graph traversal with the nodes corresponding to the maximal polyhedra and edges corresponding to their one-codimensional facets. To pass from one maximal cone to its neighboring cones, we compute tropical links using an algorithm described in \cite{HR18}, which relies on triangular decomposition and Puiseux expansions.
Our algorithm is implemented in a framework for massively parallel computations in computer algebra \cite{BDFPRR}, which is based on the idea of separation of computation and coordination. For the coordination layer, it relies on the workflow management system \textsc{GPI-Space} \cite{GPISpace} to model parallel algorithms in terms of Petri nets, while the computational layer is built on \textsc{Singular}.

Furthermore, we develop a first algorithm towards computing \emph{positive} tropicalizations as studied by Speyer and Williams \cite{SpeyerWilliams05}. While tropical varieties arise from solutions of systems of polynomial equations, positive tropical varieties arise from positive real solutions of systems of polynomial equations. They relate to graphical models in algebraic statistics \cite{PS04} and, more recently, the tropicalization of semialgebraic sets in non-archimedian semidefinite programming \cite{AGS16,JSY18}. It is also conjectured \cite[Conjecture 8.1]{SpeyerWilliams05} that they encode the combinatorics of cluster algebras of finite type. This was proven recently by Brodsky and Stump for many important cases \cite{BS18}. 

One class of tropical varieties of particular interest and the main example in this article are \emph{tropical Grassmannians} $\TropGrass_0(k,n)$. In algebraic geometry, Grassmannians $\Grass(k,n)$ parametrize all $k$-dimensional linear spaces in $K^n$ for a given field $K$. In tropical geometry, their tropicalizations $\TropGrass_0(k,n)$ parametrize all $k$-dimensional tropical linear spaces in $\RR^n$ that are realizable over $K$. In both algebraic and tropical geometry, Grassmannians form one of the simplest moduli spaces and offer a strong basis for the understanding of general (tropical) varieties. In more applied context, the real points on $\Grass(k,n)$ and their tropicalization on $\TropGrass_0(k,n)$ are linked to the soliton solutions of the Kadomtsev-Petviashvili equation \cite{KW13Survey} with positivity corresponding to regularity at all times \cite{KW13Regularity,KW14}.

We employ our algorithms to compute all maximal-dimensional cones in $\TropGrass^+(3,8)$ and compare them to the cluster complex of $\Grass(3,8)$. We verify that \cite[Conjecture 8.1]{SpeyerWilliams05} holds, which is proven to be true in \cite{BS18}. This serves as a verification of our computations, and as an alternative proof of the conjecture in this specific case. Recently, two groups of researchers, Arkani-Hamed, Lam and Spradlin \cite{ALS20} as well as Speyer and Williams \cite{SW20}, have independently proven that on the positive part the tropical Grassmannian $\TropGrass_0(k,n)$ equals the positive Dressian $\Dress^+(k,n)$.

This article is organized as follows: In Section~\ref{sec background}, we fix our notation on tropical geometry, and give some background material required in the subsequent sections.

In Section~\ref{sec:parallelFramework}, we present our massively parallel algorithm for computing tropical varieties with symmetry, which is described in terms of a Petri net. We discuss our implementation in the \textsc{Singular}/\textsc{GPI-Space} framework.

Continuing the work of \cite{SpeyerSturmfels:2004} and \cite{HerrmannJensenJoswigSturmfels:2009} we use our implementation to compute the tropical Grassmannian $\TropGrass_0(3,8)$, and analyze its natural fan structures using the software polymake \cite{polymake}. Thus we give a positive answer to Question~37 in \cite{HJS14} whether it is feasible to compute $\TropGrass_0(3,8)$. Furthermore, we compare it to the \emph{Dressian} $\Dress(3,8)$ as described in \cite{HJS14}. The Dressian $\Dress(k,n)$ parametrizes all $k$-dimensional tropical linear spaces in $\RR^n$, also known as valuated-matroids, independent of their realizability and is generally of higher dimension than the tropical Grassmannian $\TropGrass_p(k,n)$ it contains. We show that $\TropGrass_0(3,8)$ refines the $16$-dimensional skeleton of $\Dress(3,8)$ with exception of $23$ \emph{extended Fano cones} for which explicit obstructions for the realizability of tropical linear spaces are presented. 

In Section~\ref{sec:positiveTropicalization}, we propose general algorithms for computing all maximal-dimensional cones in a tropical variety $\Trop(I)$ which belong to the positive tropicalization $\Trop^+(I)$. These algorithms exploit the symmetry of $\Trop(I)$ even though $\Trop^+(I)$ itself need not be entirely symmetric.

In Section~\ref{sec:positiveGrassmannian}, we compute all maximal-dimensional cones in $\TropGrass^+(3,8)$ and compare them to the cluster complex of $\Grass(3,8)$, verifying that \cite[Conjecture 8.1]{SpeyerWilliams05} holds.

In Section~\ref{sec:open}, we discuss three open questions beyond the scope of this article. These are the coarsest structure on tropical varieties, positive tropicalizations and cluster complexes of infinite type, and real tropicalizations and the topology of real algebraic varieties.

All data and other auxiliary materials will be available shortly and hosted on \href{https://polymake.org}{https://polymake.org}. All polynomial data will be uploaded in \textsc{Singular} format, while all tropical data are directly available via \textsc{polymake} using \textsc{polyDB}.

\section*{Acknowledgments}
\noindent
We would like to thank Santiago Laplagne (Universidad de Buenos Aires) for his implementation of the Newton-Puisseux algorithm in the \textsc{Singular} library \texttt{puiseux.lib}, Mirko Rahn (Fraunhofer ITWM Kaiserslautern) for his invaluable support with \GPISpace, Lukas Ristau (TU Kaiserslautern) for his work on combining \Singular with \GPISpace, Christian Reinbold (TU M\"unchen) for his work on fan traversals using \Singular and \GPISpace, Christian Stump (Ruhr-Universit\"at Bochum) for his insight into cluster algebras, and Charles Wang (Harvard University) for his advice on cluster complexes. We would also like express our greatest gratitude to Fraunhofer ITWM Kaiserslautern for providing us with the necessary computing resources.


\section{Background}\label{sec background}
\noindent
In this section, we fix a few notations by briefly going over some basic concepts of immediate relevance to us. Our notation is largely compatible with that of \cite{MS15}, with the exception that we will use polynomials instead of Laurent polynomials and the \texttt{max}-convention instead of the \texttt{min}-convention. 
This is because the software that we will be presenting in the latter sections is built on infrastructure using polynomials and the \texttt{max}-convention.

\begin{convention}
  For the remainder of the section, let $K$ be an algebraically closed field with a non-trivial valuation $\nu\colon K^\ast\rightarrow\RR$, ring of integers $R$, and residue field~$\mathfrak K$. We fix a splitting $\mu\colon (\nu(K^\ast),+)\rightarrow (K^\ast,\cdot)$ and abbreviate $t^a\coloneqq \mu(a)$ for $a\in K^\ast$.\linebreak We use $\overline{(\cdot)}$ to denote the canonical projection $R\rightarrow \mathfrak K$, and we fix a multivariate polynomial ring $K[x]\coloneqq K[x_1,\ldots,x_n]$.
\end{convention}

\begin{definition}
  The \emph{initial form} of a polynomial $f=\sum_{\alpha\in\NN^n}c_\alpha x^\alpha \in K[x]$ with respect to a weight vector $w\in\RR^n$ is given by
  \begin{align*}
    \initial_w(f)\coloneqq& \sum_{w\cdot\alpha - \nu(c_\alpha) \text{ maximal}} \overline{t^{-\nu(c_\alpha)} c_\alpha} \cdot x^\alpha \in\mathfrak K[x], \\
    \intertext{whereas the \emph{initial ideal} of an ideal $I\unlhd K[x]$ with respect to $w\in\RR^n$ is given by}
    \initial_w(I)\coloneqq& \langle \initial_w(g)\mid g\in I\rangle \unlhd\mathfrak K[x].
  \end{align*}
\end{definition}

The following two equivalent definitions for tropical variety is part of the \emph{Fundamental Theorem of Tropical Algebraic Geometry} \cite[Theorem 3.2.5]{MS15}.

\begin{definition}\label{def:tropicalVariety}
  Let $I\unlhd K[x]$ be an ideal and $V(I)\subseteq K^n$ its corresponding affine variety. The \emph{tropical variety} of $I$ is defined to be
  \begin{align*}
    \Trop(I) \coloneqq\,& \text{cl}\Big(\big\{ (-\nu(z_1),\ldots,-\nu(z_n))\in \RR^n \mid (z_1,\ldots,z_n)\in V(I)\cap(K^\ast)^n \big\}\Big)\\
    =\,& \phantom{\text{cl}\Big(}\big\{ w\in \RR^n\mid \initial_w(I) \text{ contains no monomial}\big\}.
  \end{align*}
  where $\text{cl}(\cdot)$ denotes the closure in the euclidean topology.
\end{definition}


\begin{example}
  Let $K=\CC\{\!\{t\}\!\}$ be the field of complex Puiseux series and $\nu$ its natural valuation. Consider the linear ideal $I=\langle x+y+1\rangle \subseteq K[x,y]$. Figure~\ref{fig:fundamentalTheorem} shows $\Trop(I)$ using both definitions, with valuations resp. weight vectors highlighted.
\end{example}

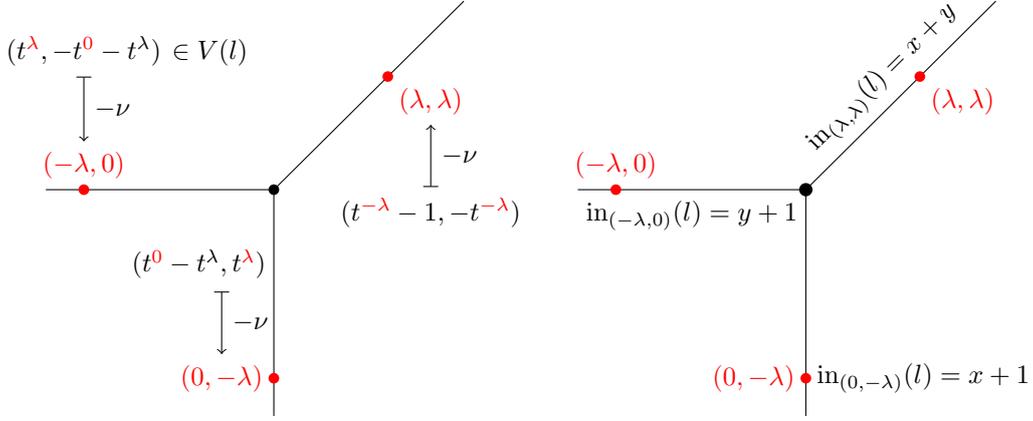
\begin{figure}[h]
  \centering
  \begin{tikzpicture}[every node/.style={font=\footnotesize}]
    \node (o1) at (-3.25,0) {};
    \fill (o1) circle (2pt);
    \draw (o1.center) -- ++(-3,0)
    (o1.center) -- ++(0,-3)
    (o1.center) -- ++(2.5,2.5);

    \node (o11) at ($(o1.center)+(-2.5,0)$) {};
    \fill[red] (o11.center) circle (2pt);
    \node[red,anchor=south] (w1) at (o11.center) {$(-\lambda,0)$};
    \node[anchor=base,yshift=15mm] (p1) at (w1.base) {$(t^{\textcolor{red}{\lambda}},-t^{\textcolor{red}{0}}-t^{\lambda})$};
    \draw[|->] (p1) -- node[right] {$-\nu$} (w1);
    \node[xshift=-0.15cm,anchor=base west] at (p1.base east) {$\in V(l)$};

    \node (o10) at ($(o1.center)+(1.5,1.5)$) {};
    \fill[red] (o10.center) circle (2pt);
    \node[red,anchor=north west] (w0) at (o10.center) {$(\lambda,\lambda)$};
    \node[anchor=base,yshift=-15mm] (p0) at (w0.base) {$(t^{\textcolor{red}{-\lambda}}-1,-t^{\textcolor{red}{-\lambda}})$};
    \draw[|->] (p0) -- node[right] {$-\nu$} (w0);

    \node (o12) at ($(o1.center)+(0,-2.5)$) {};
    \fill[red] (o12.center) circle (2pt);
    \node[red,anchor=east] (w2) at (o12.center) {$(0,-\lambda)$};
    \node[anchor=base,xshift=-3mm,yshift=15mm] (p2) at (w2.base) {$(t^{\textcolor{red}{0}}-t^{\lambda},t^{\textcolor{red}{\lambda}})$};
    \draw[|->] (w2)++(0,1.15) -- node[right] {$-\nu$} (w2);

    \node (o2) at (3.75,0) {};
    \fill (o2.center) circle (2.5pt);
    \draw (o2.center) -- node[below] {$\initial_{(-\lambda,0)}(l)=y+1$} ++(-3,0)
    (o2.center) -- ++(0,-3)
    (o2.center) -- node[sloped, above] {$\initial_{(\lambda,\lambda)}(l)=x+y$} ++(2.5,2.5);

    \node (o20) at ($(o2.center)+(1.5,1.5)$) {};
    \fill[red] (o20) circle (2pt);
    \node[red,anchor=north west] at (o20) {$(\lambda,\lambda)$};

    \node (o21) at ($(o2.center)+(-2.5,0)$) {};
    \fill[red] (o21) circle (2pt);
    \node[red,anchor=south] at (o21) {$(-\lambda,0)$};

    \node (o22) at ($(o2.center)+(0,-2.5)$) {};
    \fill[red] (o22) circle (2pt);
    \node[red,anchor=east] at (o22) {$(0,-\lambda)$};
    \node[anchor=west] at (o22) {$\initial_{(0,-\lambda)}(l) = x+1$};
  \end{tikzpicture}\vspace{-3mm}
  \caption{$\Trop(\langle l\rangle)$ defined two ways with $l=x+y+1$.}
  \label{fig:fundamentalTheorem}
\end{figure}

Tropical geometry usually involves Laurent polynomials $K[x^\pm]$ and takes place in the algebraic torus $(K^\ast)^n$. When working with polynomials, it is therefore important to assume all ideals to be saturated with respect to the product of all variables, or \emph{saturated} in short.

\begin{theorem}[Structure Theorem for Tropical Varieties {\cite[Theorem 3.3.6]{MS15}}]
  Let $I\unlhd K[x]$ be saturated and prime of dimension $d$. Then $\Trop(I)$ is the support of a balanced polyhedral complex, pure of dimension $d$, connected in codimension $1$.
\end{theorem}

\begin{definition}
  We define the \emph{Gr\"obner polyhedron} of a homogeneous ideal $I\unlhd K[x]$ around a weight vector $w\in\RR^n$ to be
  \[ C_w(I)\coloneqq \text{cl}\Big(\{v\in\RR^n\mid \initial_v(I)=\initial_w(I)\}\Big). \]
  The \emph{Gr\"obner complex} $\Sigma(I)$ is the set of all Gr\"obner polyhedra of $I$.
\end{definition}

\begin{proposition}[{\cite[Theorem 2.5.3]{MS15}}]\label{prop:structureTheorem}
  Let $I\unlhd K[x]$ be a homogeneous ideal. Then $C_w(I)$ is a closed convex polyhedron for all $w\in\RR^n$, and $\Sigma(I)$ is a finite polyhedral complex. In~particular, $\Trop(I)$ is the support of the subcomplex of the Gr\"obner fan consisting of all Gr\"obner polyhedra whose initial ideals contain no monomials.
\end{proposition}

For the sake of simplicity, we will restrict ourselves to what is commonly called the \emph{constant coefficient case}. While the parallel framework described in Section~\ref{sec:parallelFramework} works in full generality, the rest of the paper falls in this very special case.

\begin{convention}
  From now on, assume that the field $K$ is either $\overline{\QQ}\{\!\{t\}\!\}$, $\CC\{\!\{t\}\!\}$ or $\overline{\FF}_p\{\!\{t\}\!\}$ for some prime $p$, and that all ideals are both homogeneous and generated by polynomials with coefficients in $\overline{\QQ}$, $\CC$ or $\overline{\FF}_p$.

  In particular, all Gr\"obner polyhedra will be invariant under multiplication by positive real numbers, and we will be refer to Gr\"obner polyhedra and Gr\"obner complexes as \emph{Gr\"obner cones} and \emph{Gr\"obner fans} instead.

  Additionally, all Gr\"obner cones in this paper will be inside the tropical variety unless explicitly specified otherwise. This means that \emph{maximal} or \emph{maximal-dimensional} Gr\"obner cones will only be (inclusion) maximal or maximal-dimensional among the Gr\"obner cones on the tropical variety, i.e., a maximal-dimensional Gr\"obner cone is a Gr\"obner cone $C_w(I)\subseteq\Trop(I)$ with $\dim C_w(I) = \dim \Trop(I)$.

  Moreover, we will use $\Trop(I)$ to denote both the set in Definition~\ref{def:tropicalVariety} and the subfan of the Gr\"obner fan covering it by Proposition~\ref{prop:structureTheorem}. In written text, we will refer to the latter as the \emph{Gr\"obner structure} on $\Trop(I)$.
\end{convention}

Finally, let us recall the definition of Grassmannians.

\begin{definition}\label{def:Grassmannian}
  Let $1\leq k\leq n$. In the following, we will abbreviate the $n$-element set $\{1,\ldots,n\}$ by $[n]$ and the set of all $k$-element subsets of $[n]$ by $\binom{[n]}{k}$. The \emph{Grassmannian} $\Grass(k,n)$ is the variety defined by the ideal
  \[ \mathcal{I}_{k,n} \coloneqq \Big\langle \mathcal{P}_{I,J} \bigmid I\in\textstyle\binom{[n]}{k-1}, J\in\textstyle\binom{[n]}{k+1} \Big\rangle \subseteq K\Big[p_L \bigmid L\in\textstyle\binom{[n]}{k}\Big], \]
  where
  \[ \mathcal{P}_{I,J}  \coloneqq \sum_{j\in J\setminus I} (-1)^{|\SmallSetOf{i\in I}{i<j}|+|\SmallSetOf{j'\in J}{j'>j}|}\cdot p_{I\cup j}\cdot p_{J\setminus j}. \]
The ideal $\mathcal{I}_{k,n}$ is commonly referred to as \emph{Pl\"ucker ideal}, while the $\mathcal{P}_{I,J}$ are commonly called \emph{Pl\"ucker relations}. Note that $\mathcal{P}_{I,J}$ is a trinomial if and only if $|I\cap J|=k-2$, in which case we will refer to them as \emph{3-term Pl\"ucker relations}. The $3$-term Pl\"ucker relations do not generate the Pl\"ucker ideal if $n \geq d+3 \geq 6$, but they always generated the Pl\"ucker ideal up to saturation, see \cite[Section 2]{HerrmannJensenJoswigSturmfels:2009}.

  The \emph{tropical Grassmannian} is the tropicalization of the Pl\"ucker ideal, and we will denote it by $\TropGrass_p(k,n)\coloneqq\Trop(\mathcal I_{k,n})$, where $p$ is the characteristic of the field~$K$. This is well-defined, as the tropical Grassmannian only depends on the characteristic of $K$, since the coefficients of the Pl\"ucker relations are integers.
\end{definition}

Similarly to the classical Grassmannian, its tropicalization $\TropGrass_p(k,n)$ is the easiest example of a non-trivial moduli space. Each point on $\TropGrass_p(k,n)$ corresponds to the tropicalization of a $k$-dimensional linear space in the projective space $\mathbb{P}^{n-1}$.

In this article, we will mainly focus on the case $p=0$, $k=3$, and $n=8$. This is a continuation of the two articles \cite{SpeyerSturmfels:2004} and \cite{HerrmannJensenJoswigSturmfels:2009}, which discuss the tropical Grassmannians $\TropGrass_p(2,n)$, $\TropGrass_p(3,6)$ and $\TropGrass_p(3,7)$.


\section{Parallel computation of tropical varieties}\label{sec:parallelFramework}
\noindent
In this section, we discuss the realization of our algorithm for computing tropical varieties in terms of a Petri net and the technical background on massive parallelization. We begin by recalling the algorithms for computing tropical varieties.

\subsection{Computing tropical varieties}\label{sec:comptrop}

The general framework for computing tropical varieties of polynomial ideals has remained unchanged since its initial conception in \cite{BJSST07} and implementation in \textsc{Gfan} \cite{gfan}. Tropical varieties are computed by traversing all maximal Gr\"obner cones which lie on them, see Figure~\ref{fig:traversal} for an illustration. 
The traversal starts at a pre-computed starting point on the tropical variety. The inequalities and equations of each Gr\"obner cone are uniquely determined by a Gr\"obner basis with respect to a weighted ordering given by any of its relative interior points. The approach consists of two key sub-steps:
\begin{description}[leftmargin=*]
\item[Gr\"obner walk] The Gr\"obner walk \cite{CKM97,FJLT07} is a well-established technique for transforming Gr\"obner bases with respect to one ordering to that of another for a fixed ideal. To determine a tropical variety, it is used to compute Gr\"obner bases with respect to the different orderings associated to the different Gr\"obner cones of full dimension on the tropical variety.
\item[Tropical link] Tropical links describe tropical varieties locally around a point. They are tropical varieties of a combinatorially simpler type and can therefore be computed using specialized algorithms. When computing tropical varieties, they are used to identify the directions pointing towards the maximal cones neighboring a facet (of codimension one). The computation of tropical links has been the bottleneck of the algorithm for a long time. This issue has been resolved through newer approaches using projections \cite{Chan13} or root approximations \cite{HR18}. Newer versions of \textsc{gfan} rely on \cite{Chan13}, while our implementation relies on \cite{HR18}.
\end{description}

\begin{figure}
  \centering
  \begin{tikzpicture}
    \useasboundingbox (-2,2) rectangle (8,-2);
    \node[font=\small] at (-1,1.5) {$\Trop(I)$};
    \draw (0,0) -- (1,0)
    (0,0) -- (0,1)
    (0,0) -- (-1,-1)
    (1,0) -- (2,1)
    (1,0) -- (1,-1)
    (0,1) -- (1,2)
    (0,1) -- (-1,1)
    (-1,-1) -- (-2,-1)
    (-1,-1) -- (-1,-2);

    \fill[blue] (0,0) circle (1.25pt)
    (1,0) circle (1.25pt)
    (0,1) circle (1.25pt)
    (-1,-1) circle (1.25pt);
    \fill[red] (0.35,0) circle (1.5pt);
    \draw[blue,->] ($(0.25,0)+(0,-0.1)$) to[out=180,in=45] ($(-0.2,-0.2)+(0.1,-0.1)$);
    \draw[blue,->] ($(0.25,0)+(0,0.1)$) to[out=180,in=270] ($(0,0.35)+(0.1,0)$);

    \draw[blue,->] ($(0.75,0)+(0,0.1)$) to[out=0,in=225] ($(1.2,0.2)+(-0.1,0.1)$);
    \draw[blue,->] ($(0.75,0)+(0,-0.1)$) to[out=0,in=90] ($(1,-0.35)+(-0.1,0)$);

    \draw[blue,->] ($(0,0.75)+(0.1,0)$) to[out=90,in=225] ($(0.2,1.2)+(0.1,-0.1)$);
    \draw[blue,->] ($(0,0.75)+(-0.1,0)$) to[out=90,in=0] ($(-0.35,1)+(0,-0.1)$);

    \draw[blue,->] ($(-0.8,-0.8)+(0.1,-0.1)$) to[out=225,in=90] ($(-1,-1.35)+(0.1,0)$);
    \draw[blue,->] ($(-0.8,-0.8)+(-0.1,0.1)$) to[out=225,in=0] ($(-1.35,-1)+(0,0.1)$);

    \node[anchor=west,font=\tiny] at (1,0) {$u$};
    \draw[dashed] (1,0) circle (0.5cm);
    \coordinate (oo) at (5,0);
    \draw[dashed] (oo) circle (1.25cm);
    \fill[blue] (oo) circle (2pt);
    \draw[blue,->]
    (oo.center) -- ++(0.55,0.55);
    \draw[blue,->]
    (oo.center) -- ++(-0.75,0);
    \draw[blue,->]
    (oo.center) -- ++(0,-0.75);
    \draw
    ($(oo.center)+(0.6,0.6)$) -- ++(0.25,0.25)
    ($(oo.center)+(-0.8,0)$) -- ++(-0.4,0)
    ($(oo.center)+(0,-0.8)$) -- ++(0,-0.4);
    \node[anchor=west,font=\small] at (oo) {$u$};
    \node[anchor=base west,font=\small,fill=white] at ($(oo)+(0.5,-1)$) {$\Trop(\initial_{u}(I))$};
    \draw[dashed] (1,0.5) -- (4.75,1.225);
    \draw[dashed] (1,-0.5) -- (4.75,-1.225);
    \node[anchor=north west,font=\footnotesize] at (-0.5,-1.5) {tropical point to start the traversal};
    \draw[<-] (0.35,-0.25) -- ++(0,-1.25);
    \node[anchor=south west,font=\footnotesize] at (3.75,1.5) {tropical link to direct the traversal};
  \end{tikzpicture}\vspace{-2mm}
  \caption{Traversing the tropical variety.}
  \label{fig:traversal}
\end{figure}
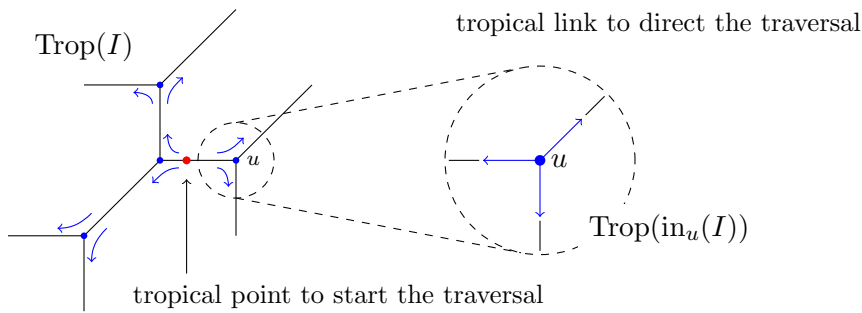

\subsection{Massive parallelization in computer {algebra}}
Our implementation builds on a framework for massively parallel computations in computer algebra \cite{BDFPRR,Ristau}, which combines the computer algebra system \Singular with the workflow management system \GPISpace. This framework originated in work on a parallel smoothness criterion for algebraic varieties \cite{BDFPRR,Ristau}, and has been used in \cite{Reinbold,BFKRR} to realize a massively parallel fan traversal  for computing GIT-fans. For an overview and more applications see \cite{BFK} and \cite{BBDGPRWZ}. The results of the current section extend the traversal of complete fans developed for the GIT-fan algorithm. More details can be found in the thesis of the first author \cite{Bendle}.

The workflow management system \GPISpace is based on the idea of separation of coordination and computation \cite{GC92}. In the coordination layer it uses the language of Petri nets \cite{Petri62} to model a computer program in the form of a concurrent system. It allows for running computations on large scientific computing clusters, and consists of the following three main components \cite[Section~4]{BDFPRR}:
\begin{enumerate}[leftmargin=*]
\item a distributed runtime system managing available resources and assigning jobs to resources,
\item a virtual memory layer allowing processes to communicate and share data,
\item a workflow manager tracking the global structure and state, formulated as a so-called Petri net.
\end{enumerate}

\begin{definition}\label{def:petri}
  A \emph{Petri net} is a finite bipartite directed graph $N=(P, T, F)$, where $P$ and $T$ are disjoint vertex sets called \emph{places} and \emph{transitions} respectively, and where the set of edges $F \subseteq (P\times T) \cup (T \times P)$ is called the set of \emph{flow relations}. Given $p \in P$ and $t \in T$, we call $p$ an \emph{input} to $t$ if $(p,t) \in F$ and $p$ an \emph{output} of $t$ if $(t,p) \in F$.
\end{definition}

Petri nets depict a static model of the algorithm, with transitions representing processes and places representing data passed between them. The dynamics of the algorithm, i.e., its execution, is described via the notion of markings:

\begin{definition}\label{def:marking}
  Let $(P,T,F)$ be a Petri net. A \emph{marking} $M$ is a map $M \colon P \to \NN_{\geq 0}$, and we say a place $p\in P$ \emph{holds $k$ tokens} under $M$ if $M(p) = k$.

  We call a transition $t \in T$ \emph{enabled} if $M(p)>0$ for all $p\in P$ with $(p,t) \in F$. In this case, the transition $t$ may be \emph{fired} by consuming one token of each input and returning one token in each output, which leads to a new marking $M'$ given by
	\[ M'(p) \coloneqq M(p) - |\{ (p,t) \} \cap F| + |\{ (t,p) \} \cap F| \]
  for all $p \in P$. We denote the firing process by $\smash{M \overset{t}\longrightarrow M'}$.
\end{definition}

One important principle to adhere in modelling the state of algorithms via markings in \GPISpace is locality: Firing enabled places should be a local process and not block other enabled transitions from firing simultaneously.
To ensure locality, \GPISpace requires the user to impose restrictions on places such that any token that is in an input to multiple transitions can only be consumed by a single well-defined transition.

\begin{figure}[b]
  \centering
  \begin{tikzpicture}[scale=0.8, node distance=1.7cm]
    \tikzset{transition/.append style={font=\ttfamily}}
    \node [place] at (0,0) (list) {$\ell$};
    \node [transition,align=center] at (-3,0) (empty)
      {consume empty\\\scriptsize if empty($\ell$)}
      edge [pre] (list);
    \node [transition,align=center] at (3,0) (split) {split\\\scriptsize if not empty($\ell$)}
      edge [pre, bend left] (list)
      edge [post,bend right] (list);
    \node [place] at (6,0) (out) {$e$} edge [pre] (split);
  \end{tikzpicture}\vspace{-3mm}
  \caption{Petri net unwrapping a list}
  \label{fig:unwrap}
\end{figure}
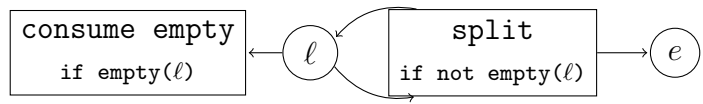

\begin{example}
  Figure~\ref{fig:unwrap} shows an example which unwraps a list of tokens. In it, both \texttt{split} and \texttt{consume empty} have input $\ell$. However, \texttt{split} only consumes non-empty lists from $\ell$, while \texttt{consume empty} only consumes empty lists. Thus there is always a single well-defined transition that can fire. \end{example}

As illustrated, we will usually give conditions in the form ``\texttt{if (not) }\textit{condition}''.

As a consequence of locality, enabled transitions may fire simultaneously. This includes single enabled transitions with enough input tokens to fire multiple times (Figure~\ref{fig:firing} top), which is called \emph{data parallelism}, and multiple enabled transitions with separate input tokens (Figure~\ref{fig:firing} bottom), which is called \emph{task parallelism}.

\begin{figure}[t]
  \centering
  \begin{tikzpicture}[yscale=0.8, node distance=1.4cm]
    \begin{scope}
      \node [place,tokens=2] (in)                {};
      \node [place,tokens=3] (in2) [above of=in] {};
      \node [transition]     (f)   [right of=in,yshift=0.7cm] {$t$}
        edge [pre] (in)
        edge [pre] (in2);
      \node [place]          (out) [right of=f]  {} edge [pre] (f);
    \end{scope}

    \begin{scope}[xshift=5cm]
      \node [place]          (in')                {};
      \node [place,tokens=1] (in2') [above of=in'] {};
      \node [transition]     (f')   [right of=in',yshift=0.7cm] {$t$}
        edge [pre] (in')
        edge [pre] (in2');
      \node [place,tokens=2] (out') [right of=f']  {} edge [pre] (f');
    \end{scope}

    \draw [->,thick] ([xshift=5mm]out -| out) -- ([xshift=-5mm]f' -| in')
      node[above=1mm,midway,text centered] {$t^2$};
  \end{tikzpicture}\\[3mm]
  \begin{tikzpicture}[yscale=0.8, node distance=1.4cm]
    \begin{scope}
      \node [place]      (in) {};
      \node [transition] (s)  [right of=in] {$s$} edge [pre] (in);

      \node [place,tokens=1] (su) [above right of=s] {}      edge [pre] (s);
      \node [transition] (f)  [right of=su]      {$t_1$} edge [pre] (su);
      \node [place]      (l)  [right of=f]       {}      edge [pre] (f);

      \node [place,tokens=1] (sd) [below right of=s] {}      edge [pre] (s);
      \node [transition] (g)  [right of=sd]      {$t_2$} edge [pre] (sd);
      \node [place]      (r)  [right of=g]       {}      edge [pre] (g);
    \end{scope}
    \begin{scope}[xshift=7.5cm]
      \node [place]      (in) {};
      \node [transition] (s)  [right of=in] {$s$} edge [pre] (in);

      \node [place]      (su) [above right of=s] {}      edge [pre] (s);
      \node [transition] (f)  [right of=su]      {$t_1$} edge [pre] (su);
      \node [place,tokens=1] (l)  [right of=f]   {}      edge [pre] (f);

      \node [place]      (sd) [below right of=s] {}      edge [pre] (s);
      \node [transition] (g)  [right of=sd]      {$t_2$} edge [pre] (sd);
      \node [place,tokens=1] (r)  [right of=g]   {}      edge [pre] (g);
    \end{scope}
    \draw [->,thick] (5.75,0) -- node[above=1mm] {$t_1t_2$} ++(1.25,0);
  \end{tikzpicture}\vspace{-1mm}
  \caption{Data parallelism (top) and task parallelism (bottom).}
  \label{fig:firing}
\end{figure}
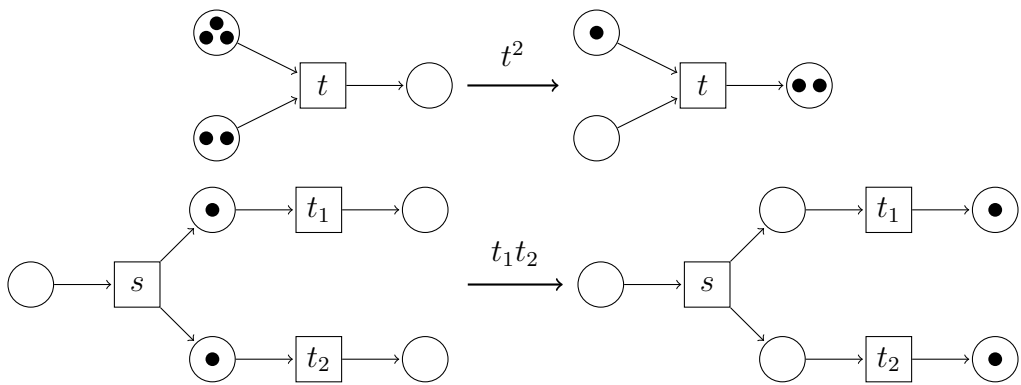

Markings are only a very basic tool for describing the dynamics of algorithms. There are many extensions of Petri nets realized in \textsc{GPI-Space} such as allowing tokens to carry data (\emph{colored} Petri nets), and to allow for transitions to take some time to execute (\emph{timed} Petri nets).

\subsection{Parallel traversals of tropical varieties}

\begin{figure}[b]
  \centering
  \begin{tikzpicture}[yscale=0.8, node distance=1.4cm]
    \tikzset{transition/.append style={font=\ttfamily}}

    \node [transition,align=center] (getC) at (2.1,-2.5) {get cone\\\footnotesize if $e_1=$\texttt{true}};
    \node [place] (cone) at (4.2,-2.5) {} edge [pre] (getC);
    \node [transition] (facets) at (6.3,-2.5) {facets} edge [pre] (cone);
    \node [place] (f) at (8.4,-2.5) {} edge [pre] (facets);
    \node [transition,align=center] (insF) at (10.5,-2.5) {store\\ facet} edge [pre] (f);

    \node [transition,align=center] (getF) at (10.5,2.5) {get facet\\\footnotesize if $e_2=$\texttt{true}};
    \node [place] (ufac) at (8.4,2.5) {} edge [pre] (getF);
    \node [transition] (neigh) at (6.3,2.5) {neighbors} edge [pre] (ufac);
    \node [place] (uC) at (4.2,2.5) {$m$} edge [pre] (neigh);
    \node [transition,align=center] (insC) at (2.1,2.5) {store\\ cone} edge [pre] (uC);

    \node [database] (cs) at (0,0) {cones}
    edge [out=270,in=180,dotted,post,thick] (getC)
    edge [out=90,in=180,dotted,pre,thick] (insC);

    \node [database] (fs) at (12.6,0) {facets}
    edge [out=90,in=0,dotted,post,thick] (getF)
    edge [out=270,in=0,dotted,pre,thick] (insF);

    \node [place,blue] (eC) at (2.1,0) {$e_1$}
    edge [out=100,in=260,pre,blue] (insC)
    edge [out=80,in=280,post,blue] (insC)
    edge [out=260,in=100,pre,blue] (getC)
    edge [out=280,in=80,post,blue] (getC);
    \node [place,blue] (eF) at (10.5,0) {$e_2$}
    edge [out=100,in=260,pre,blue] (getF)
    edge [out=80,in=280,post,blue] (getF)
    edge [out=260,in=100,pre,blue] (insF)
    edge [out=280,in=80,post,blue] (insF);

    \node [transition,align=center] (init) at (4.2,0) {starting\\ cone} edge [post] (uC);
    \node [place,tokens=1] (ctrl) at (4.2,-1.4) {} edge [post] (init);
    \node [place] (data) at (6.3,0) {$I$}
    edge [post,dashed] (neigh)
    edge [post,dashed] (facets)
    edge [post,dashed] (init);
  \end{tikzpicture}\vspace{-1mm}
  \caption{Petri net modelling the traversal of a tropical variety}
  \label{fig:refined}
\end{figure}
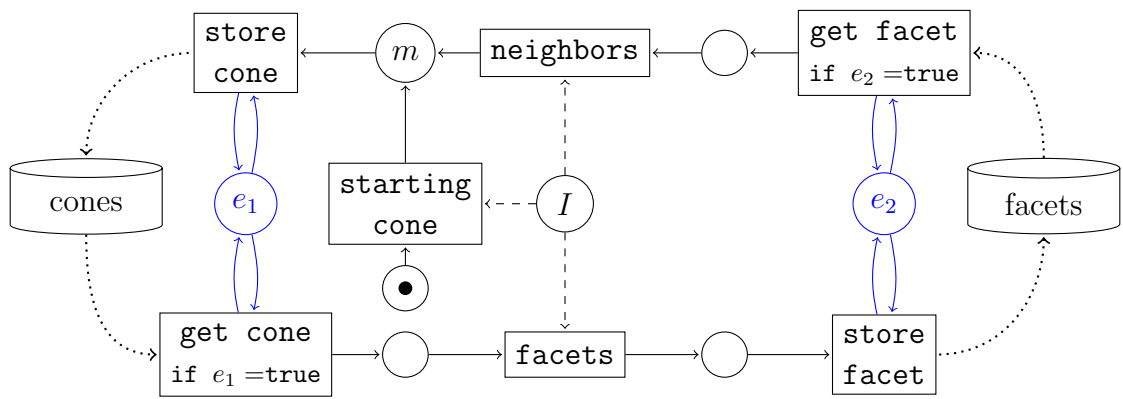

Figure~\ref{fig:refined} shows a Petri net modelling the traversal of the tropical variety as mentioned in Section~\ref{sec:comptrop}. The computation of the facets requires Gr\"obner bases, while the computation of the neighbors requires tropical links. They are separated for finer control over job sizes, see \cite{Bendle} for more details. The Petri net consists of the following places and transitions:

\begin{description}[leftmargin=*]
  \item[Place $I$] read-only input data containing ideal and symmetries.
  \item[Transition \texttt{starting cone}] computes a random maximal Gr\"obner cone on the tropical variety and places it into $m$.
  \item[Place $m$] holds maximal Gr\"obner cones on the tropical variety.
  \item[Transition \texttt{store cone}] takes a Gr\"obner cone and inserts it into the external storage, discarding cones that are already known. Afterwards, replaces the boolean token in $e_1$ by \texttt{true} or \texttt{false} depending on whether unprocessed cones remain in the storage.
  \item[Place $e_1$] holds \texttt{true} if storage contains unprocessed cones, \texttt{false} otherwise.
  \item[Transition \texttt{get cone}] retrieves an unprocessed Gr\"obner cone and marks it as processed, provided that $e_1$ holds a \texttt{true}-token.
  \item[Transition \texttt{facets}] computes the facets of a Gr\"obner cone.
  \item[Transitions \texttt{store facet} and \texttt{get facet}] analogous to \texttt{store cone} and \texttt{get cone}.
  \item[Place $e_2$] analogous to $e_1$.
  \item[Transition \texttt{neighbors}] takes a facet and computes all incident maximal Gr\"obner cones on the tropical variety.
\end{description}

The algorithm terminates if $e_1$ and $e_2$ hold a \texttt{false}-token, i.e.\ no unprocessed cones and facets remain in storage, and if the input places to \texttt{store cones} and \texttt{store facet} are empty.

\subsection{Timings for the tropical Grassmannian}

\begin{table}[b]
  \footnotesize
  \begin{subtable}[t]{0.475\textwidth}
    \centering
    \begin{tabular}[t]{r|r||r|r||r}
      \multicolumn{5}{c}{\boldmath$\TropGrass_0(3,7)$} \\ \hline
      \textbf{nodes} & \textbf{cores} & \textbf{time [s]} & \textbf{speedup} & \textbf{eff.} \\ \hline
      1 & 1   & 792.8 &  1.000 & 1.000 \\
      1 & 2   & 382.8 &  2.070 & 1.035 \\
      1 & 4   & 191.1 &  4.147 & 1.037 \\
      1 & 8   &  98.1 &  8.080 & 1.001 \\
      1 & 12  &  74.1 & 10.691 & 0.891 \\
      1 & 16  &  58.0 & 13.653 & 0.853 \\
      2 & 24  &  42.8 & 18.522 & 0.772 \\
      2 & 32  &  39.7 & 19.942 & 0.623
    \end{tabular}
  \end{subtable}\hfill%
  \begin{subtable}[t]{0.475\textwidth}
    \centering
    \begin{tabular}[t]{r|r||r|r||r}
      \multicolumn{5}{c}{\boldmath$\TropGrass_0(3,8)$} \\ \hline
      \textbf{nodes} & \textbf{cores} & \textbf{time [s]} & \textbf{speedup} & \textbf{eff.} \\
      \hline
      1  &  15 & 98926.1 &  *15.000 &*1.000 \\
      2  &  30 & 37398.7 &   39.675 & 1.322 \\
      4  &  60 & 14486.3 &  102.435 & 1.707 \\
      8  & 120 &  6597.3 &  224.925 & 1.874 \\
      16 & 240 &  3297.9 &  449.955 & 1.874 \\
      24 & 360 &  2506.0 &  592.125 & 1.645 \\
      32 & 480 &  2001.7 &  741.285 & 1.544 \\
      40 & 600 &  1509.6 &  982.935 & 1.638 \\
      48 & 720 &  1267.3 & 1170.855 & 1.626 \\
      56 & 840 &  1188.2 & 1248.825 & 1.487
    \end{tabular}
  \end{subtable}
  \caption{Timings for computing the tropical Grassmannians $\TropGrass_0(3,7)$ and $\TropGrass_0(3,8)$ in parallel}
  \label{tab:timings}
\end{table}

Table~\ref{tab:timings} shows the timings to compute the tropical Grassmannians $\TropGrass_0(3,7)$ and $\TropGrass_0(3,8)$ in relation to the number of CPU cores used. Additionally, \texttt{speedup} lists the ratios between the single- and multi-core computation times, while \texttt{efficiency} further divides that number by the number of cores. Note that, due to the size of $\TropGrass_0(3,8)$, no single core computation could be carried out. All \texttt{efficiency} numbers instead are based on the 15 core timing (marked with $^*$ in Table~\ref{tab:timings}).

All computations were run on a cluster at the Fraunhofer ITWM \cite{ClusterHardware}. Each node is fitted with two Intel Xeon E5-2670 processors and 64\,GB of memory, amounting to 16 CPU cores per node at a base clock of 2{,}6\,Ghz. The computations of $\TropGrass_0(3,8)$ were done with a fixed configuration of 15 compute jobs and one storage interface job per node.

As shown in Table~\ref{tab:timings}, computing $\TropGrass_0(3,7)$ scales favorably up to around 12 CPU cores, after which a noticeable drop in efficiency can be observed. This is expected, however, as $\TropGrass_0(3,7)$ is up to symmetry covered by only 125 cones, and the number of cores exceed the maximum queue size.

\begin{figure}[tb]
  \centering
  \small
  \includegraphics[scale=0.85]{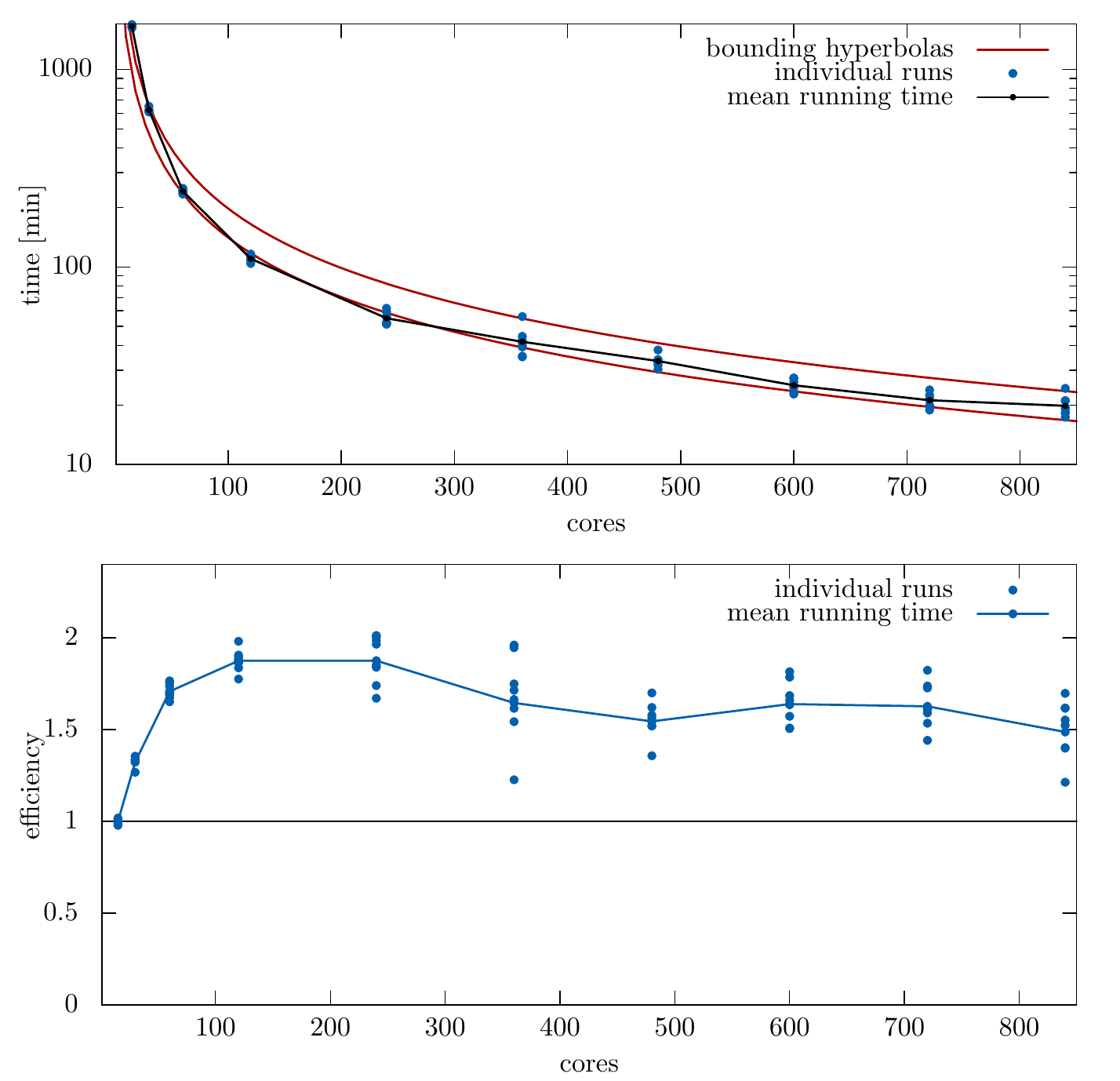}
  \caption{Timings and efficiency of $\TropGrass_0(3,8)$}
  \label{fig:timing38}
\end{figure}

Figure~\ref{fig:timing38} shows the timings and the efficiency graph for $\TropGrass_0(3,8)$. The timings scale very well, with no significant drop in efficiency, even to more than 800 cores. We do not encounter the queue size problem due to the size of $\TropGrass_0(3,8)$, and there is no visible decrease in efficiency for increasing core count, as one could expect due to increased overhead and communication.

We observe a surprising surge in efficiency around 60 cores.  At the time of writing, no comprehensive explanation has been found for this behaviour. We suspect that this is partly due to technical effects of the cluster hardware, e.g. with regard to the memory bandwidth. In \cite{BDFPRR}, experiments with a different algorithm on the same cluster have shown that distributing the number of used processor cores over more machines can lead to a speedup. This might indicate a memory bottleneck, and could partially explain the unexpected behaviour of the efficiency graph. Moreover, in \cite{BDFPRR} it was observed that massively parallel implementations can lead to a superlinear speedup by allowing the randomized algorithm to find a faster path to the final result. In our setting this amounts to different possible expansions of the tropical variety leading to different computation times.


\section{Tropical Grassmannians and Dressians}\label{sec:grass38}

\noindent
In this section, we compare the tropical Grassmannian $\TropGrass_0(3,8)$ to the Dressian $\Dress(3,8)$ described in \cite{HJS14}, i.e., we compare the moduli of realizable tropical linear spaces or valuated matroids with the moduli of all tropical linear spaces or valuated matroids. 
The main difficulty stems from the fact that both are covered by thousands of cone orbits with respect to the sizeable group $S_8$.
We begin with a brief introduction and some formal definitions.

\begin{definition}
  Let $1\leq k\leq n$ and recall 3-term Pl\"ucker relations $\mathcal P_{I,J}$ from Definition~\ref{def:Grassmannian}. The \emph{Dressian} $\Dress(k,n)$ is the intersection of their tropical hypersurfaces:
	\[ \Dress(k,n) \coloneqq  \bigcap_{I,\, J} \Trop(P_{I,J}). \]
  where the intersection is taken over all sets $I\in\binom{[n]}{k-1}$, $J\in\binom{[n]}{k+1}$ with $|I\cap J|=k-2$.
  By definition, the Dressian is the support of the common refinement of the Gr\"obner subfans covering $\Trop(P_{I,J})$, and we will use $\Dress(k,n)$ to denote both the set and the polyhedral fan covering the set.
\end{definition}

The Dressian is a \emph{tropical prevariety}, i.e., it is the intersection of the tropical hypersurfaces of a finite generating set. Usually, tropical prevarieties and tropical varieties have little in common besides one being trivially contained in the other. In fact, merely testing both objects for equality is a hard task \cite{Theobald06,GRS17}, and it is unclear what distinguishes a generating set for whom equality holds, commonly called a \emph{tropical basis} \cite{JoswigSchroeter:2018}.
Nevertheless, the Dressian is interesting for many reasons:
\begin{itemize}[leftmargin=*]
	\item The Dressian is the moduli space of all tropical linear spaces, also known as valuated matroids. Similar to the Grassmannian in algebraic geometry, the Dressian can be regarded as one of the simplest moduli spaces in tropical geometry. 
	\item The hypersimplex $\Delta(k,n)$ is the moment polytope for the torus action on the complex Grassmannian. The Dressian $\Dress(k,n)$ is the subfan of the secondary fan of the $\Delta(k,n)$ consisting of all of matroid subdivisions \cite{GelfandGoreskyMacPhersonSerganova:1987,MS15}.
\item Recent work of Huh and Br\"and\'en \cite[Theorem 8.7]{BrandenHuh19} regard the Dressian $\Dress(d,n)$ as the tropicalization of the space of Lorentzian polynomials supported on $\Delta(d,n)$, i.e., on the basis of the uniform matroid of rank $d$ on $n$ elements.
\end{itemize}

As mentioned above, the Dressian is a subfan of the secondary fan of the hypersimplex and depends only on two parameters. We will consider the Dressian equipped with this fan structure which is also known as \emph{Pl\"ucker structure}; see \cite{OlartePanizzutSchroeter:2019}.
This is the polyhedral structure that we obtain from the intersection of the tropical hypersurfaces of the $3$-term Pl\"ucker relations.
This structure is the coarsest structure, as for any two vectors that lie in distinct maximal cones there is a tropical $3$-term Pl\"ucker relation whose maximum is attained twice, but on different terms. Thus a positive combination of these vectors attains the maximum only at a single term.

The tropical Grassmannian $\TropGrass_p(d,n)$ depends on the characteristic $p$ of the underling field though almost all tropical Grassmannians agree with $\TropGrass_0(d,n)$.
As a tropical variety it is naturally a subfan of a Gr\"obner fan. In contrast to the Pl\"ucker structure the Gr\"obner structure typically is not a coarsest structure.

There are many explicit computational results on the Dressian and the tropical Grassmannian inside it.
In particular, all existing computations verify that the Pl\"ucker structure of the Dressian coarsens the Gr\"obner fan structure of the tropical Grassmannian in characteristic $p=0$, i.e., that there is a subfan of the Dressian supported on the tropical Grassmannian. This subfan is often strictly coarser than the Gr\"obner fan restricted to it.
The following remarks and Table~\ref{tab:sumT} summarize these results by comparing tropical Grassmannians with the Gr\"obner and the Pl\"ucker structures for all known $(k,n)$ and $p$ using their rays and maximal cones, the bold numbers representing the number of orbits with respect to the natural $S_n$-action on the coordinates.

\begin{table}[bt]
	\centering
	\renewcommand{\arraystretch}{0.95}
	\begin{tabular*}{0.85\linewidth}{@{\extracolsep{\fill}}llllr@{\hspace{0.05cm}}r@{\hspace{0.3cm}}r@{\hspace{0.05cm}}r}
		\toprule
		$d$ & $n$ & char. & stru. & \# rays & \# orbits & \# max. cones & \# orbits\\
		\midrule
		$3$ & $6$ & any & G & $65$ & $\mathbf{3}$ & $1035$ & $\mathbf{7}$\\
		$3$ & $6$ & any & D & $65$ & $\mathbf{3}$ & $1005$ & $\mathbf{7}$\\
		$3$ & $7$ & $p=2$ & G & $751$ & $\mathbf{7}$ & $252\,420$ & $\mathbf{125}$\\
		$3$ & $7$ & $p\neq2$ & G & $721$ & $\mathbf{6}$ & $252\,000$ & $\mathbf{125}$\\
		$3$ & $7$ & $p\neq2$ & D & $616$ & $\mathbf{5}$ & $211\,365$ & $\mathbf{94}$\\
		\bottomrule
	\end{tabular*}
	\caption{The number of rays and maximal cones of the tropical Grassmannians $\TropGrass_p(3,6)$ and $\TropGrass_p(3,6)$ with the Gr\"obner structure (G) inherited from the Gr\"obner fan or the coarsest Pl\"ucker structure (D) inherited from the Dressian.}
	\label{tab:sumT}
\end{table}

\begin{remark}
	The tropical Grassmannian $\TropGrass_p(2,n)$ with the Gr\"obner structure is independent of the characteristic, has $2^{n-1}-n-1$ rays in $\lceil\frac{n-3}{2}\rceil$ orbits and $(2n-5)!!$ maximal cones.
	The number of $S_n$-symmetry classes equals the number of trivalent trees with $n$ leaves \cite[A000672]{OEIS}.
	It is the moduli spaces of tropical lines, phylogenetic trees with $n$ labeled leaves, and tropical rational curves of genus $0$ with $n$ marked points. This structure is the coarsest fan structure. The tropical Grassmannian and Dressian agree.
	The rays correspond to split hyperplanes and the tropical Grassmannian is isomorphic to the split complex.
	The connected matroids in the corresponding matroid subdivisions are sparse matroids. See \cite{SpeyerSturmfels:2004}, \cite{HerrmannJoswig:2008} and \cite{JoswigSchroeter:2017} for further details.
\end{remark}

\begin{remark}
	The tropical Grassmannian $\TropGrass_p(3,6)$ is independent of the characteristic, the Dressian and tropical Grassmannian have the same support, but the Gr\"obner structure is a refinement of the Dressian.
	To be precise, there is a cone over an three dimensional bipyramid in the Pl\"ucker structure which the Gr\"obner structure refines into three cones over tetrahedra. The tropical Grassmannian $\TropGrass_p(3,6)$ with the Gr\"obner structure is a simplicial fan.

	The tropical Grassmannian $\TropGrass_p(3,7)$ depends on the characteristic.
	The Gr\"obner structure on  $\TropGrass_2(3,7)$ is coarse, but not a subfan of the Dressian $\Dress(3,7)$.
	The Gr\"obner structure on $\TropGrass_p(3,7)$ is a refinement of the Pl\"ucker structure if $p\neq 2$.
	The number of rays and maximal cones is summarized in Table \ref{tab:sumT}.
	Further details can be derived from \cite{SpeyerSturmfels:2004} and \cite{HerrmannJensenJoswigSturmfels:2009}.
\end{remark}

We extend the previous result by combining our computation in Section~\ref{sec:parallelFramework} with the following computational result in \cite{HJS14} describing the Dressian $\Dress(3,8)$.

\begin{proposition}[{\cite[Theorem 31]{HJS14}\footnote{Our computation showed that the data of \cite{HJS14} misses a $16$-dimensional simplicial cone of orbit size $840$. The orbit is represented by a cone containing the corank vector of the sparse-paving matroid with non-bases $015$, $024$, $067$, $126$, $137$, $235$, $346$, $457$. It is a cone whose eight rays correspond to vertex splits all lying in the same $S_8$-orbit.}
}]
	The Dressian $\Dress(3,8)$ is a non-pure $17$-dimensional fan with a $8$-dimensional lineality space, it consists of $15\,470$ rays in twelve $S_8$-orbits and $117\,595\,485$ cones of dimension $16$ in $4\,789$ $S_8$-orbits.
\end{proposition}

The following theorem is a summary of a elaborate computation based on \Singular, \Polymake and the framework described in Section~\ref{sec:parallelFramework}.

\begin{theorem}\label{thm:TropGrass38GroebnerStructure}
  The Gr\"obner subfan supported on the tropical Grassmannian $\TropGrass_0(3,8)$ is a $16$-dimensional fan with a $8$-dimensional lineality space, it consists of $732\,725$ rays in $95$ $S_8$-orbits and $278\,576\,760$ maximal cones in $14\,763$ $S_8$-orbits.

	Moreover, the coarsest fan structure supported on $\TropGrass_0(3,8)$ is a subfan of the Dressian consisting of $15\,470$ rays in twelve $S_8$-orbits and $117\,445\,125$ maximal cones in $4\,766$ $S_8$-orbits.
\end{theorem}
\begin{proof}
  We computed the tropical Grassmannian $\TropGrass_0(3,8)$  with the Gr\"obner structure on with the methods of Section~\ref{sec:parallelFramework}.

  In order to confirm that Pl\"ucker structure on $\TropGrass_0(3,8)$ is well-defined, we tested that the relative interior of any $16$-dimensional Dressian cone is either contained in or disjoint to $\TropGrass_0(3,8)$.
  For that, we verified that any maximal Gr\"obner cone on $\TropGrass_0(3,8)$ is contained in a $16$-dimensional Dressian cone, and that every $15$-dimensional Gr\"obner cone on $\TropGrass_0(3,8)$ intersecting the relative interior of a $16$-dimensional Dressian cone is contained in exactly two maximal Gr\"obner cones.
\end{proof}

\begin{remark}\label{rem:cubic}
	There are $23$ $S_8$-orbits of $16$-dimensional cones in the Dressian $\Dress(3,8)$ whose relative interior does not intersect the tropical Grassmannian $\TropGrass_0(3,8)$. Remarkably, the following single polynomial is a witness for all Dressian orbits:
	\begin{align*}
    f\coloneqq\,&2\,p_{123} p_{467} p_{567} - p_{367} p_{567} p_{124} - p_{167} p_{467} p_{235} - p_{127} p_{567} p_{346}
		- p_{126} p_{367} p_{457}\\ &\qquad - p_{237} p_{467} p_{156} + p_{134} p_{567} p_{267} + p_{246} p_{567} p_{137} + p_{136} p_{267} p_{457}
		\in \mathcal{I}_{3,7}\subset \mathcal{I}_{3,8}.
	\end{align*}
  To be precise, for any of the aforementioned $23$ orbits $\mathcal D\subseteq\Dress(3,8)$ there exists a Dressian cone $\sigma\in\mathcal D$ such that $\initial_w(f)=2\cdot p_{123} p_{467} p_{567}$ for all relative interior points $w\in\text{Relint}(\sigma)$.

	This observation agrees with \cite[Proposition 5.5]{HampeJoswigSchroeter:2019}, which states that regular matroid subdivisions that contain matroid polytopes of extensions of the Fano matroid lead to points in the Dressian that are not on the tropical Grassmannian. Our witness polynomial $f$ is in fact a witness of the higher-dimensional Fano cone, whose $16$-dimensional faces generate the $23$ $S_8$-orbits.
\end{remark}

As an immediate consequence, we get:

\begin{theorem}\label{thm:TropGrass38TropicalBasis}
	The quadratic Pl\"ucker relations and the cubic polynomials in the $S_8$-orbit of $f$ are a tropical basis of the Pl\"ucker ideal $\mathcal I_{3,8}$. 
\end{theorem}
\begin{proof}
  By definition, the Pl\"ucker relations generate the Pl\"ucker ideal $\mathcal I_{3,8}$. Recall that $\Dress(3,8)$ is $17$-dimensional. By \cite[Remark 5.4]{HampeJoswigSchroeter:2019}, the polynomials in $S_8\cdot f$ are witnesses for all $17$-dimensional cones of $\Dress(3,8)$, i.e., for every point $w$ inside a $17$-dimensional cone of $\Dress(3,8)$ there is a $\sigma\in S_8$ with $w\notin\Trop(\sigma\cdot f)$. In Theorem~\ref{thm:TropGrass38GroebnerStructure} we verified that any $16$-dimensional cone of $\Dress(3,8)$ either lies on $\TropGrass_0(3,8)$ or has a relative interior disjoint to it. In Remark~\ref{rem:cubic}, we verified that the polynomials in $S_8\cdot f$ are witnesses for the latter.
\end{proof}

During our computations, we also encountered the following phenomena, which will be relevant for Section~\ref{sec:TropGrassPlus38}. We conjecture for them to hold for arbitrary $k$ and $n$ in characteristic $0$:

\begin{theorem}\label{thm:TropGrass38Saturation}
  For any $w,v\in\TropGrass_0(3,8)$ we have
  \[w \text{ and } v \text{ lie on the same cone of } \Dress(3,8) \quad \Longleftrightarrow\quad \initial_w(\mathcal I_{3,8}):p^\infty = \initial_v(\mathcal I_{3,8}):p^\infty, \]
  where $(\cdot):p^\infty$ denotes the saturation on the product of all Pl\"ucker variables.

\end{theorem}
\begin{proof}
  The statement was proven through explicit computations in \textsc{Singular}. Note that it suffices to verify that weight vectors in the same Dressian cone have the same saturated initial ideal, because weight vectors in different Dressian cones have different saturated initial ideals:

  Let $w,v\in\TropGrass_p(k,n)$ be in two distinct Dressian cones, which means there exist a three term Pl\"ucker relation $\mathcal P=s_0+s_1+s_2$ such that $\initial_w(\mathcal P)\neq \initial_v(\mathcal P)$, and assume that $\initial_w(\mathcal I_{k,n}):p^\infty = \initial_v(\mathcal I_{k,n}):p^\infty$.
	There are two cases that might appear.
	In the first case $\initial_w(\mathcal P)\neq \mathcal P$ and $\initial_v(P)\neq \mathcal P$. Let us assume $\initial_w(\mathcal P)= s_0+s_1$ and $\initial_v(P)=s_0+s_2$, then $s_0\in \initial_w(\mathcal I_{k,n}):p^\infty = \initial_v(\mathcal I_{k,n}):p^\infty$ contradicting that $\initial_w(\mathcal I_{k,n})$ and $\initial_v(\mathcal I_{k,n})$ are monomial free.
The second case is  $\initial_w(\mathcal P)=\mathcal P$ or $\initial_v(P)= \mathcal P$, say $\initial_w(\mathcal P)= s_0+s_1+s_2$ and $\initial_v(P)=s_0+s_2$, then $s_1\in \initial_w(\mathcal I_{k,n}):p^\infty = \initial_v(\mathcal I_{k,n}):p^\infty$ again contradicting that $\initial_w(\mathcal I_{k,n})$ and $\initial_v(\mathcal I_{k,n})$ are monomial free.
\end{proof}

In general, the Gr\"obner structure on a tropical variety is far from being as coarse as possible, which incurs many iterations in the traversal of the tropical variety that might seem unnecessary. For example, there is a maximal cone in the tropical Grassmannian $\TropGrass_0(3,8)$ equipped with the Pl\"ucker structure that is refined into $2620$ maximal cones with the Gr\"obner structure.
Therefore, an important open question is whether the tropical variety can be equipped with a natural polyhedral structure that is coarser than that of the Gr\"obner fan.

Theorem~\ref{thm:TropGrass38Saturation} states that saturated initial ideals provides such a structure for the tropical Grassmannian $\TropGrass_0(k,n)$ for $k\leq 3$ and $n\leq 8$. While it is unknown whether it holds for all tropical Grassmannians over fields of characteristic $0$, the result does not generalize to arbitrary tropical varieties, as neither the set of all weight vectors which share the same saturated initial ideal nor its euclidean closure need be convex:

\begin{example}\label{ex:tropvar} Consider the homogeneous ideal
	\[ I \coloneqq  \langle x(a-b) - y(c-d), x(c-d) - y(a-b)\rangle \unlhd \CCt[a,b,c,d,x,y]. \]
	It is not hard to see that its tropical variety $\Trop(I)$ in $\RR^6$ is four-dimensional with a two-dimensional lineality space generated by $(1,1,1,1,0,0)$ and $(0,0,0,0,1,1)$. Figure~\ref{fig:singularCode} contains a quick \Singular computation which shows that the Gr\"obner subfan supported on $\Trop(I)$ consists of $12$ rays and $16$ maximal cones. Figure~\ref{fig:tropvar} furthermore illustrates the combinatorial structure of $\Trop(I)$. It shows the intersection $\Trop(I)\cap \text{Lin}(e_1,e_2,e_3,e_5)$, which removes the two-dimensional lineality space and whose maximal cones are contained in either $\text{Lin}(e_1,e_2,e_3)$ or $\text{Lin}(e_1+e_2,e_5)$. In other words, the following equality holds both in terms of sets and polyhedral complexes:
  \begin{align*}
    \Trop(I) &= \Big(\big(\Trop(I)\cap \text{Lin}(e_1,e_2,e_3)\big)\cup\big(\Trop(I)\cap \text{Lin}(e_1+e_2,e_5)\big)\Big)\\
             &\hspace{3cm} +\text{Lin}(e_1+e_2+e_3+e_4,e_5+e_6).
  \end{align*}
	The intersection $\Trop(I)\cap \text{Lin}(e_1,e_2,e_3)$ resembles the standard tropical plane in $\RR^3$ with two opposing maximal cones barycentric subdivided, while the intersection $\Trop(I)\cap \text{Lin}(e_1+e_2,e_5)$ resembles $\RR^2$ divided into octants. In the first intersection all saturated initial ideals with respect to relative interior points of maximal cones are generated by $a-b$ and $c-d$, whereas in the second intersection the saturated initial ideals with respect to relative interior points of maximal cones are distinct.
\end{example}

\lstdefinestyle{singular}{
  language=C,
  basicstyle=\scriptsize\ttfamily,
  moredelim=**[is][\color{orange}]{@o}{@o},
  moredelim=**[is][\color{blue}]{@b}{@b},
  moredelim=**[is][\color{red}]{@r}{@r},
}
\begin{figure}[t]
\begin{lstlisting}[style=singular]
> ring r = 0,(a,b,c,d,x,y),dp;
> ideal I = x*(a-b)-y*(c-d), x*(c-d)-y*(a-b);
> LIB "tropical.lib";
> tropicalVariety(I);
RAYS                            MAXIMAL_CONES
-1  0  0  0   0 0 # 0           @b{0 2}@b # Dimension 4
-1 -1  0  0  -1 0 # 1           @o{1 2}@o
@r-1 -1  0  0   0 0 # 2@r           @o{2 3}@o
-1 -1  0  0   1 0 # 3           @b{0 7}@b
 0  0  0  0  -1 0 # 4           @b{0 11}@b
 0  0  0  0   1 0 # 5           @o{1 4}@o
 0 -1  0  0   0 0 # 6           @o{3 5}@o
 0  0 -1  0   0 0 # 7           @b{2 6}@b
 1  1  0  0  -1 0 # 8           @o{4 8}@o
 @r1  1  0  0   0 0 # 9@r           @o{5 10}@o
 1  1  0  0   1 0 # 10          @b{6 7}@b
 1  1  1  0   0 0 # 11          @b{6 11}@b
                                @b{7 9}@b
LINEALITY_SPACE                 @o{8 9}@o
 1  1  1  1   0 0 # 0           @o{9 10}@o
 0  0  0  0   1 1 # 1           @b{9 11}@b
  \end{lstlisting}\vspace{-3mm}
  \caption{\Singular code for Example~\ref{ex:tropvar} (output cleaned up for clarity)}
  \label{fig:singularCode}
\end{figure}


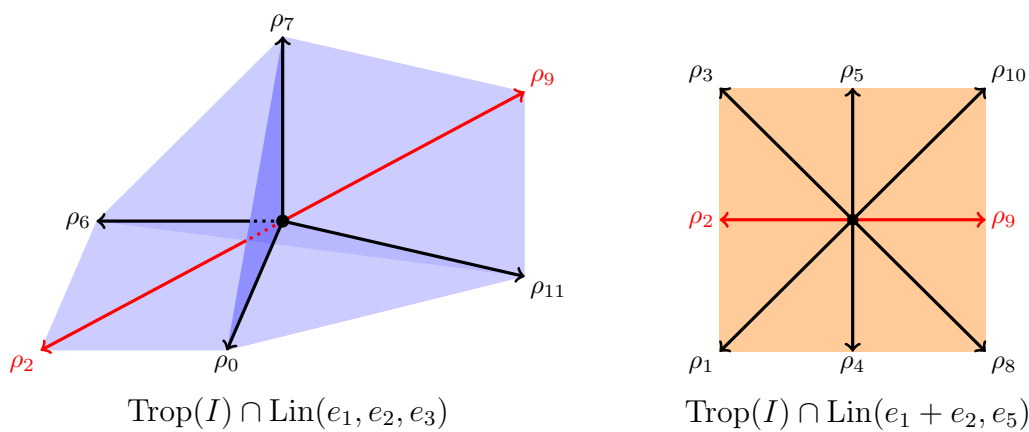
\begin{figure}[t]
  \centering
  \begin{tikzpicture}
    \node (T123) at (0,0)
    {%
      \begin{tikzpicture}[scale=2.45,x={(-0.3cm,-0.7cm)},y={(-1cm,0cm)},z={(0cm,1cm)},font=\footnotesize,inner sep=1pt,outer sep=1pt]
        \coordinate (o) at (0,0,0);
        \coordinate (r0) at ($(o)+(1,0,0)$);
        \coordinate (r2) at ($(o)+(1,1,0)$);
        \coordinate (r6) at ($(o)+(0,1,0)$);
        \coordinate (r7) at ($(o)+(0,0,1)$);
        \coordinate (r9) at ($(o)+(-1,-1,0)$);
        \coordinate (r11) at ($(o)+(-1,-1,-1)$);
        \fill[blue,opacity=0.2] (o) -- (r0) -- (r2) -- cycle;
        \fill[blue,opacity=0.3] (o) -- (r0) -- (r7) -- cycle;
        \fill[blue,opacity=0.2] (o) -- (r0) -- (r11)-- cycle;
        \fill[blue,opacity=0.2] (o) -- (r2) -- (r6) -- cycle;
        \fill[blue,opacity=0.2] (o) -- (r6) -- (r7) -- cycle;
        \fill[blue,opacity=0.1] (o) -- (r6) -- (r11)-- cycle;
        \fill[blue,opacity=0.2] (o) -- (r7) -- (r9) -- cycle;
        \fill[blue,opacity=0.2] (o) -- (r9) -- (r11)-- cycle;
        \draw[->,very thick] (o) -- (r0) node[anchor=north] {$\rho_0$};
        \draw[dotted,very thick,red] (o) -- ++(0.15,0.15);
        \draw[->,very thick,red] ($(o)+(0.15,0.15)$) -- (r2) node[anchor=north east] {$\rho_2$};
        \draw[dotted,very thick] (o) -- ++(0,0.175);
        \draw[->,very thick] ($(o)+(0,0.175)$) -- (r6) node[anchor=east] {$\rho_6$};
        \draw[->,very thick] (o) -- (r7) node[anchor=south] {$\rho_7$};
        \draw[->,very thick,red] (o) -- (r9) node[anchor=south west] {$\rho_9$};
        \draw[->,very thick] (o) -- (r11) node[anchor=north west] {$\rho_{11}$};
        \fill (o) circle (1pt);
      \end{tikzpicture}%
    };
    \node (T125) at (7.5,-0.35) {%
      \begin{tikzpicture}[scale=1.75,font=\footnotesize,x={(-1cm,0cm)},y={(0cm,-1cm)},inner sep=1pt,outer sep=1pt]
        \coordinate (o) at (0,0,0);
        \coordinate (r1) at ($(o)+(1,1)$);
        \coordinate (r2) at ($(o)+(1,0)$);
        \coordinate (r3) at ($(o)+(1,-1)$);
        \coordinate (r4) at ($(o)+(0,1)$);
        \coordinate (r5) at ($(o)+(0,-1)$);
        \coordinate (r8) at ($(o)+(-1,1)$);
        \coordinate (r9) at ($(o)+(-1,0)$);
        \coordinate (r10) at ($(o)+(-1,-1)$);
        \fill[orange!40] ($(o)+(-1,-1)$) rectangle ($(o)+(1,1)$);
        \draw[->,very thick] (o) -- (r1) node[anchor=north east] {$\rho_1$};
        \draw[->,very thick,red] (o) -- (r2) node[anchor=east] {$\rho_2$};
        \draw[->,very thick] (o) -- (r3) node[anchor=south east] {$\rho_3$};
        \draw[->,very thick] (o) -- (r4) node[anchor=north] {$\rho_4$};
        \draw[->,very thick] (o) -- (r5) node[anchor=south] {$\rho_5$};
        \draw[->,very thick] (o) -- (r8) node[anchor=north west] {$\rho_8$};
        \draw[->,very thick,red] (o) -- (r9) node[anchor=west] {$\rho_9$};
        \draw[->,very thick] (o) -- (r10) node[anchor=south west] {$\rho_{10}$};
        \fill (o) circle (1.25pt);
      \end{tikzpicture}
    };
    \node[anchor=north] at (T123.south) {$\Trop(I)\cap\text{Lin}(e_1,e_2,e_3)$};
    \node[anchor=north] at (T125.south) {$\Trop(I)\cap\text{Lin}(e_1+e_2,e_5)$};
  \end{tikzpicture}\vspace{-3mm}
  \caption{Tropical variety in Example~\ref{ex:tropvar}}\label{fig:tropvar}
\end{figure}


\section{Computing positive tropicalizations}\label{sec:positiveTropicalization}
\noindent
In this section, we recall the notion of positive tropicalization by Speyer and Williams \cite{SpeyerWilliams05}, and we introduce algorithms for testing which maximal-dimensional Gr\"obner cones lie in the positive tropicalization. These algorithms exploit the symmetry of the tropical variety even though the positive tropical variety inside of it may not be symmetric with respect to it.

We distinguish between cones whose initial ideals are binomial and cones whose initial ideals are not. For binomial cones, we state a simple combinatorial algorithm. For non-binomial cones, we reduce the problem to dimension zero which can then be tackled symbolically, numerically or with a mix of both.

\begin{convention}\label{con:positivity}
  For the remainder of the section, let $K\coloneqq \CC\{\!\{t\}\!\}$ be the field of complex Puiseux series and fix an ideal $I\unlhd K[x]\coloneqq K[x_1,\dots,x_n]$ that is generated over the subfield $R\subseteq K$ of Puiseux series whose lowest coefficient is real:
	\[ R\coloneqq \Big\{\sum_{\alpha\geq \lambda} c_\alpha t^{\alpha}\in K\bigmid 0\neq c_{\lambda}\in \RR \Big\} \cup \{0\}. \]
	In particular, any Gr\"obner basis of $I$ will consist of polynomials with coefficients in~$R$ and any initial ideal of $I$ will be generated over $\RR[x]$. We denote by $R_{>0}$ the set of complex Puiseux series whose lowest term is real and positive,

  Suppose $I$ is invariant under a group $S$ acting on $K[x]$ via signed permutation of the variables, i.e., for each group element $\sigma \in S$ and all variables $x_i\in K[x]$ there is a permutation $|\sigma|\in S_n$ and a sign $u_i\in\{\pm 1\}$ with $\sigma\cdot x_i = u_i\cdot x_{|\sigma|(i)}$. This implies that $V(I)$ is invariant under $S$ acting on $K^n$ via signed permutation of the components, and that $\Trop(I)$ is invariant under $S$ acting on $\RR^n$ via unsigned permutation of the components.
\end{convention}

\begin{definition}\label{def:positiveTropicalization}
  We define the \emph{positive tropicalization} of an ideal $I\unlhd K[x]$ to be
	\[ \Trop^+(I) \coloneqq \text{cl}\Big(\nu\big(V(I)\cap (R_{>0})^n\big)\Big)\subseteq\RR^n, \]
  where again $\nu(\cdot)$ denotes componentwise valuation and $\text{cl}(\cdot)$ denotes the closure in the euclidean topology.

  For the sake of convenience, we call a weight vector $w\in\RR^n$, an initial ideal $\initial_w(I)\unlhd\CC[x]$, and a Gr\"obner cone $C_w(I)\subseteq \Trop(I)$ \emph{positive} if $w\in\Trop^+(I)$.
\end{definition}

Note that under the Fundamental Theorem of Tropical Geometry, positive tropical varieties also admit an algebraic description:

\begin{proposition}[{\cite[Proposition 2.2]{SpeyerWilliams05}}]\label{prop:SpeyerWilliams}
  Let $I\unlhd K[x]$ be an ideal. Then
  \[ \Trop^+(I) = \Big\{w\in\RR^n \bigmid \initial_w(I)\text{ monomial free and } \initial_w(I)\cap \RR_{\geq 0}[x] = \langle 0\rangle\Big\}. \]
  In particular, $\Trop^+(I)$ is covered by all positive Gr\"obner cones if $I$ is homogeneous.
\end{proposition}

As an easy corollary, we get that positivity only depends on the saturated initial ideals, which will be relevant for Section~\ref{sec:TropGrassPlus38}.

\begin{corollary}\label{cor:positivityAndSaturation}
  Let $I\unlhd K[x]=K[x_1,\dots,x_n]$ be an ideal and let $w,v\in\RR^n$ be two weight vectors with $\initial_w(I):(\prod_{i=1}^n x_i)^\infty=\initial_v(I):(\prod_{i=1}^n x_i)^\infty$. Then
  \[ w\in \Trop^+(I)\quad \Longleftrightarrow\quad v\in \Trop^+(I)\]
\end{corollary}
\begin{proof}
  The statement follows directly from the following two facts:
  \begin{itemize}[leftmargin=*]
  \item $\initial_w(I)$ is monomial free if and only if $\initial_w(I):(\prod_{i=1}^n x_i)^\infty\neq\langle1\rangle$,
  \item $\initial_w(I)\cap \RR_{\geq 0}[x] = \langle 0\rangle$ if and only if $\initial_w(I):(\prod_{i=1}^n x_i)^\infty\cap \RR_{\geq 0}[x] = \langle 0\rangle$. \qedhere
  \end{itemize}
\end{proof}

\subsection{Binomial cones}
We decide positivity of binomial cones using the description of Proposition~\ref{prop:SpeyerWilliams}. We begin by recalling a well-known result on the Gr\"obner bases of binomial ideals, and derive an easy test for positivity of binomial ideals from it.

\begin{proposition}[{\cite[Proposition 1.1]{EisenbudSturmfels96}}]\label{prop:binomialGroebnerBasis}
  Any reduced Gr\"obner basis of a binomial ideal consists solely of binomials.
\end{proposition}

\begin{lemma}\label{lem:positiveBinomialGroebnerBasis}
  Let $J\unlhd \RR[x]$ be a non-trivial binomial ideal, $G\subseteq J$ a reduced Gr\"obner basis with respect to any ordering $>$. Then
  \[ J\cap \RR_{\geq 0} [x] = \langle 0\rangle \quad\Longleftrightarrow\quad G\cap \RR_{\geq 0} [x] = \emptyset. \]
\end{lemma}
\begin{proof}
  \begin{description}[leftmargin=*]
  \item[$\Rightarrow$] Trivial, as $G\subseteq J$ and $0\not\in G$.
  \item[$\Leftarrow$] By Proposition~\ref{prop:binomialGroebnerBasis}, the Gr\"obner basis $G$ consists solely of binomials. And by definition, each element of $G$ is normalized. So let $G$ contain only normalized binomials whose non-leading coefficient is negative. Then the S-polynomial of any polynomial $f\in  \RR[x]$ with respect to a Gr\"obner basis element $g\in G$,
    \[ \spoly_>(f,g) \coloneqq \frac{\lcm(\LM_>(f),\LM_>(g))}{\LM_{>}(f)}\cdot f - \frac{\lcm(\LM_>(f),\LM_>(g))}{\LM_{>}(g)}\cdot g,\]
    will preserve the parity of $f$, i.e., if $f\in \RR_{\geq 0}[x]$ then also $\spoly_>(f,g)\in \RR_{\geq 0}[x]$, possibly $\spoly_>(f,g)=0$.

    Now assume there is a non-zero polynomial $f\in J\cap \RR_{\geq 0} [x]$. As $G$ is a Gr\"obner basis, dividing $f$ with respect to $G$ will yield remainder $0$. However, the division with respect to $G$ is merely a nested chain of S-polynomials of $f$ with respect to a sequence $(g_1,\dots,g_r)$ of possibly repeating elements $g_i\in G$:
    \[ \spoly_>(\underbrace{\spoly_>(\dots\spoly_>(f,g_1)\dots,g_{r-1})}_{\eqqcolon f_r\neq 0},g_r)=0. \]
    Abbreviating the penultimate non-zero S-polynomial with $f_r$, this implies two things: First, as $\spoly_>(f_r,g_r)=0$, $f_r$ must be a multiple of $g_r$.
    Second, because spoly preserves the parity of $f$ and $f\in \RR_{\geq 0}[x]$, we also have $f_r\in \RR_{\geq 0}[x]$.
    Both together contradict that $G$ contains only binomials with a positive and a negative coefficient.\qedhere
  \end{description}
\end{proof}

\begin{proposition}\label{prop:positivityBinomial}
  Let $C_w(I)\subseteq\Trop(I)$ be a maximal cone with $\initial_w(I)$ binomial, and let $G\subseteq\initial_w(I)$ be a reduced Gr\"obner basis with respect to any ordering $>$. Then for any element $\sigma\in S$ we have
  \[ \sigma\cdot C_w(I)\subseteq\Trop^+(I)\quad\Longleftrightarrow\quad \sigma \cdot G\cap \RR_{\geq 0} [x] = \emptyset \text{ and } \sigma\cdot G\cap \RR_{\leq 0}[x]=\emptyset. \]
\end{proposition}
\begin{proof}
  As $G\subseteq \initial_w(I)$ is a reduced Gr\"obner basis with respect to the ordering $>$, $\sigma\cdot G\subseteq \initial_{\sigma\cdot w}(I)$ will be a Gr\"obner basis with respect to the ordering $>_\sigma$ defined by
  \[ x^\alpha >_\sigma x^\beta\quad:\Longleftrightarrow\quad x^{\sigma\cdot\alpha} > x^{\sigma\cdot \beta}, \]
  where $\sigma$ acts on the exponent vectors as it does on the weight space $\RR^n$.

  By Proposition~\ref{prop:binomialGroebnerBasis}, $G$ consists solely of binomials and hence so does $\sigma\cdot G$. Moreover, $\sigma\cdot G$ is reduced up to normalization. The claim then follows from Lemma~\ref{lem:positiveBinomialGroebnerBasis}.
\end{proof}

\begin{example}
  Consider the Grassmannian $\Grass(2,5)$, whose Pl\"ucker ideal $I$ is generated by three $3$-term Pl\"ucker relations:
  \begin{align*}
    I \coloneqq  \langle & p_{12}p_{34}-p_{13}p_{24}+p_{14}p_{23},\, p_{02}p_{34}-p_{03}p_{24}+p_{04}p_{23},\, p_{01}p_{34}-p_{03}p_{14}+p_{04}p_{13},\\
                 & p_{01}p_{24}-p_{02}p_{14}+p_{04}p_{12},\, p_{01}p_{23}-p_{02}p_{13}+p_{03}p_{12} \rangle \unlhd \CC\{\!\{t\}\!\}[p_{ij}\!\mid\! 0\!\leq\!i\!<\!j\!\leq\! 4].
  \end{align*}
	The Petersen Graph in Figure~\ref{fig:Grass25} illustrates the combinatorics of the tropical Grassmannian $\TropGrass_0(2,5)$ modulo its $5$-dimensional lineality space generated by
	\[ \sum_{\substack{0\leq i<j\leq 4\\i\neq k\neq j}} e_{ij} \quad \text{for } k=0,\ldots,4, \]
	where $e_{ij}$ denotes the unit vector in direction of $p_{ij}$ in the weight space. Each vertex denotes a ray generated by the negative of the inscribed unit vector, and each edge denotes a maximal cone spanned by two rays. The edges in red are the maximal cones inside the positive tropical Grassmannian $\TropGrass^+(2,5)$.
  \begin{figure}[th]
    \centering
    \begin{tikzpicture}[font=\scriptsize]
      \useasboundingbox (-5,-3) rectangle (5,3.5);
      \node[draw,circle,inner sep=2pt] (e01) at (90:30mm) {$e_{01}$};
      \node[draw,circle,inner sep=2pt] (e23) at (162:30mm) {$e_{23}$};
      \node[draw,circle,inner sep=2pt] (e04) at (234:30mm) {$e_{04}$};
      \node[draw,circle,inner sep=2pt] (e12) at (306:30mm) {$e_{12}$};
      \node[draw,circle,inner sep=2pt] (e34) at (12:30mm) {$e_{34}$};
      \node[draw,circle,inner sep=2pt] (e24) at (90:13mm) {$e_{24}$};
      \node[draw,circle,inner sep=2pt] (e14) at (162:13mm) {$e_{14}$};
      \node[draw,circle,inner sep=2pt] (e13) at (234:13mm) {$e_{13}$};
      \node[draw,circle,inner sep=2pt] (e03) at (306:13mm) {$e_{03}$};
      \node[draw,circle,inner sep=2pt] (e02) at (12:13mm) {$e_{02}$};
      \draw[red,very thick] (e01) -- node (w) {} (e23) -- node (w14) {} (e04) -- (e12) -- (e34) -- (e01);
      \draw[thick] (e01) -- node (w34) {} (e24)
      (e23) -- (e14)
      (e04) -- (e13)
      (e12) -- (e03)
      (e34) -- (e02)
      (e13) -- (e24) -- (e03) -- (e14) -- (e02) -- (e13);
      \fill[red] (w) circle (2.5pt);
      \node[above] at (w) {$w$};
      \fill[red] (w14) circle (2.5pt);
      \fill (w34) circle (2.5pt);
      \path (w) edge[dashed,bend right=30,->] node[below,xshift=-1mm,yshift=0.5mm,font=\tiny] {$(34)$} (w34);
      \path (w) edge[dashed,out=135,in=200,min distance=30mm,->] node[left,xshift=0.5mm,font=\tiny] {$(14)$} (w14);
    \end{tikzpicture}\vspace{-3mm}
    \caption{The tropical Grassmannian $\TropGrass^+(2,5)$ and its positive cones.}
    \label{fig:Grass25}
  \end{figure}
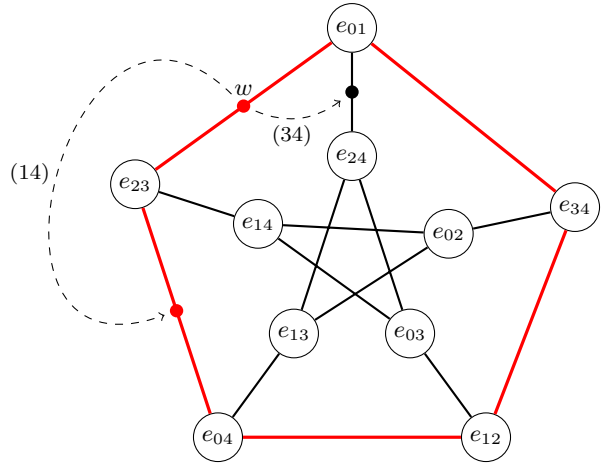

  The weight vector $w\coloneqq e_{01}+e_{23}$ lies in the interior of a maximal cone $C_w(I)$. Its corresponding initial ideal $\initial_w(I)$ is generated by the following binomial reduced Gr\"obner basis:
  \[ G\coloneqq \{
    p_{02}p_{13}-p_{12}p_{03},\,
    p_{02}p_{14}-p_{12}p_{04},\,
    p_{02}p_{34}-p_{03}p_{24},\,
    p_{03}p_{14}-p_{13}p_{04},\,
    p_{12}p_{34}-p_{13}p_{24} \}. \]
  Thus, according to Lemma~\ref{lem:positiveBinomialGroebnerBasis}, $w$ is contained in $\TropGrass^+(2,5)$.

  Moreover, consider the two transpositions $(14),(34)\in S_5$, which act on the coordinate ring as follows:
  \begin{align*}
    (14):
    \begin{cases}
      p_{01}\mapsto p_{04}, & p_{13}\mapsto -p_{34},\\
      p_{02}\mapsto p_{02}, & p_{14}\mapsto -p_{14},\\
      p_{03}\mapsto p_{03}, & p_{23}\mapsto p_{23},\\
      p_{04}\mapsto p_{01}, & p_{24}\mapsto -p_{12},\\
      p_{12}\mapsto -p_{24}, & p_{34}\mapsto -p_{13},
    \end{cases}
                          &&\text{and}&&
    (34):
    \begin{cases}
      p_{01}\mapsto p_{01}, & p_{13}\mapsto p_{14},\\
      p_{02}\mapsto p_{02}, & p_{14}\mapsto p_{13},\\
      p_{03}\mapsto p_{04}, & p_{23}\mapsto p_{24},\\
      p_{04}\mapsto p_{03}, & p_{24}\mapsto p_{23},\\
      p_{12}\mapsto p_{12}, & p_{34}\mapsto -p_{34}.
    \end{cases}
  \end{align*}
  Applying them to $G$ yields
  \[
    (14)\cdot G =
    \left\{
      \begin{array}{l}
        -p_{02}p_{34}+p_{24}p_{03},\\
        -p_{02}p_{14}+p_{24}p_{01},\\
        -p_{02}p_{13}+p_{03}p_{12},\\
        -p_{03}p_{14}+p_{34}p_{01},\\
        p_{24}p_{13}-p_{34}p_{12}
      \end{array}
    \right\} \quad\text{and}\quad
    (34)\cdot G =
    \left\{
      \begin{array}{l}
        p_{02}p_{14}-p_{12}p_{04},\\
        p_{02}p_{13}-p_{12}p_{03},\\
        -p_{02}p_{34}-p_{04}p_{23},\\
        p_{04}p_{13}-p_{14}p_{03},\\
        -p_{12}p_{34}-p_{14}p_{23}
      \end{array}
    \right\}\]
  Hence, by Proposition~\ref{prop:positivityBinomial}, $(14)\cdot w=e_{04}+e_{23}$ lies on the positive tropical Grassmannian $\TropGrass^+(2,5)$, whereas $(34)\cdot w=e_{01}+e_{24}$ does not.
\end{example}

From Proposition~\ref{prop:positivityBinomial}, we obtain the following simple algorithm:

\begin{algorithm}[Positivity of binomial cones]\label{alg:positivityBinomial}\
  \begin{algorithmic}[1]
    \REQUIRE{$(G,S)$, where
      \begin{itemize}[leftmargin=*]
      \item $G\subseteq\initial_w(I)\unlhd \CC[x]$, a reduced Gr\"obner basis of a binomial initial ideal $\initial_w(I)$,
      \item $S$, a group as in Convention~\ref{con:positivity}.
      \end{itemize}}
    \ENSURE{$S^{+}_w(I)\coloneqq \{\sigma\in S\mid \sigma\cdot C_w(I)\subseteq\Trop^{+}(I)\}$, a set of symmetries which map $C_w(I)$ into $\Trop^{+}(I)$.}
    \RETURN{$\bigcap_{g\in G} \{ \sigma\in S\mid \sigma\cdot g \text{ has coefficients with mixed parity}\}$}
  \end{algorithmic}
\end{algorithm}

\begin{remark}\label{rem:radicalMultiplicityOneCone}
  By \cite[Proof of Lemma 1]{BLMM17}, any maximal $C_w(I)\subseteq\Trop(I)$ of multiplicity one has a primary initial ideal $\initial_w(I)$ and a binomial radical $\sqrt{\initial_w(I)}$. And since
  \[ \Big(\initial_w(I) \text{ positive} \Longleftrightarrow \sqrt{\initial_w(I)} \text{ positive}\Big) \quad\text{and}\quad \sigma\cdot\sqrt{\initial_w(I)} = \sqrt{\initial_{\sigma\cdot w}(I)} \;\;\forall\sigma\in S, \]
  we can use Algorithm~\ref{alg:positivityBinomial} to test positivity within their orbit. 
\end{remark}

\subsection{Algorithm for general cones}
We decide positivity of general maximal Gr\"obner cones using Definition~\ref{def:positiveTropicalization}. The idea is to reduce the problem to dimension zero, for which we can explicitly compute the signs of the finite number of roots.

We begin with recalling a central lemma for the proof of the Fundamental Theorem of Tropical Geometry, which allows us to read off positivity from the zeroes of the initial ideal.

\begin{lemma}[{\cite[Proposition 3.2.11]{MS15}}]\label{lem:initialZeroes}
  Let $w=(w_1,\ldots,w_n)\in\Trop(I)$, then
	\[
		V(\initial_w(I)) \cap (\CC^*)^n \!=\!
		\{ (\overline{t^{-w_1}z_1},\ldots,\overline{t^{-w_n}z_n})\mid (z_1,\ldots,z_n)\!\in\! V(I)\cap{(\CC^*)^n} \!\text{ with } \nu(z_i)\!=\!w_i\}.
	\]
	In particular, $C_w(I)\subseteq\Trop^+(I)$ if and only if $V(\initial_w(I))\cap \RR_{> 0}^n\neq\emptyset$.
\end{lemma}
\begin{proof}
  \begin{description}[leftmargin=*]
  \item[$\subseteq$] This is \cite[Proposition 3.2.11]{MS15}.
  \item[$\supseteq$]
	  Let $z\coloneqq (z_1,\ldots,z_n)\in V(I)$ with $\nu(z)=w$. Then for any $f\in I$ we have $f(z)=0$ by definition, which necessarily implies $\initial_w(f)(\overline{t^{-w_1}z_1},\ldots,\overline{t^{-w_n}z_n})=0$. Hence $z\in V(\initial_w(I))$. \qedhere
  \end{description}
\end{proof}

The next lemma allows us to reduce the problem to dimension zero.

\begin{lemma}\label{lem:positivityHomogeneity}
  Let $J\unlhd \RR[x]$ be weighted homogeneous with respect to a weight vector $0\neq w=(w_1,\ldots,w_n)\in\ZZ^n$, say $w_i\neq 0$. Then
  \[ V(J)\cap(\RR_{>0})^n\neq\emptyset\quad\Longleftrightarrow\quad V(J + \langle x_i-1\rangle)\cap (\RR_{>0})^n\neq\emptyset \]
  and moreover $\dim(J + \langle x_i-1\rangle)=\dim(J)-1$.
\end{lemma}
\begin{proof}
  \begin{description}[leftmargin=*]
  \item[$\Leftarrow$] Clear, as $V(J+\langle x_i-1\rangle)\subseteq V(J)$.
  \item[$\Rightarrow$] Note that the weighted homogeneity of $J$ induces a torus action
  \[ \CC^\ast \times V(J)\longrightarrow V(J), \quad (a,(z_1,\ldots,z_n)) \longmapsto (a^{-w_1}z_1,\ldots,a^{-w_n} z_n),  \]
  with $w_i\neq 0$. Hence for any $z\in V(J)\cap(\RR_{>0})^n$ there exists an $a\in\CC^\ast$ with $a^{-w_i}\cdot z_i=1$. \qedhere
  \end{description}
\end{proof}

\goodbreak

\begin{algorithm}[Positivity reduced to dimension $0$]\label{alg:positivityReduction}\
  \begin{algorithmic}[1]
    \REQUIRE{$G\subseteq\initial_w(I)$, a reduced Gr\"obner basis for a maximal cone $C_w(I)\subseteq\Trop(I)$.}
    \ENSURE{$H$, generators of a zero-dimensional ideal $J\unlhd \CC[x_1,\ldots,x_n]$ such that
	  \[ C_w(I)\subseteq\Trop^+(I) \quad\Longleftrightarrow\quad  J\cap(\RR_{>0})^n\neq\emptyset.\]}\vspace{-1em}
    \STATE Compute a basis $b_1,\dots,b_d\in\RR^n$ of the $d$-dimensional vector subspace
    \[ C_0(\initial_w(I)) = \{ v\in\RR^n \mid \initial_v(g)=g \text{ for all } g\in G\}\subseteq \RR^n \]
    such that the matrix $B\in\RR^{d\times n}$ with rows $b_1,\dots,b_d$ is in row-echelon form.
    \STATE Let $\Lambda\subseteq \{1,\dots,n\}$ denote the column-indices of the pivots of $B$.
	  \RETURN{$H\coloneqq G\cup\SmallSetOf{x_i-1}{i\in \Lambda}$.}
  \end{algorithmic}
\end{algorithm}
\begin{proof}[Proof of correctness]
  By Lemma~\ref{lem:initialZeroes}, we have $C_w(I)\subseteq\Trop^{+}(I)$ if and only if $V(\initial_w(I))\cap(\RR_{>0})^n\neq\emptyset$.

	By \cite[Proposition 2.4]{BJSST07}, $C_0(\initial_w(I))$ is the set of all vectors with respect to whom $\initial_w(I)$ is weighted homogeneous. We can thus apply Lemma~\ref{lem:positivityHomogeneity} iteratively $d$ times to obtain $V(\initial_w(I))\cap(\RR_{>0})^n\neq\emptyset$ if and only if $V(J)\cap(\RR_{>0})^n\neq\emptyset$.
\end{proof}

Additionally, we require an algorithm for computing the signs of the roots of a zero-dimensional ideal. We will treat this part as a black box, and discuss various possibilities in Remark~\ref{rem:positivitySigns}.

\begin{algorithm}[Black box algorithm for determining sign]\label{alg:positivitySigns}\
  \begin{algorithmic}[1]
    \REQUIRE{$H\subseteq J$, a generating set of a zero-dimensional ideal $J\unlhd \CC[x]$.}
    \ENSURE{$R\subseteq \{\pm 1\}^n$, such that
      \[ R=
        \begin{cases}
          \sgn(V(H)) &\text{if }V(H)\subseteq(\RR^\ast)^n,\\
          \emptyset &\text{otherwise}
        \end{cases} \]
      where $\sgn(\cdot)$ denotes the map that is componentwise
    \[ \RR^\ast \longrightarrow \{\pm 1\}, \quad z \longmapsto
      \begin{cases}
        +1 &\text{if }z>0\\
        -1 &\text{if }z<0
      \end{cases}\]}
  \end{algorithmic}
\end{algorithm}

Combining Algorithms~\ref{alg:positivityReduction} and \ref{alg:positivitySigns}, we obtain Algorithm~\ref{alg:positivityGeneral} for positivity within orbits of maximal cones.

\begin{algorithm}[Positivity of maximal-dimensional cones]\label{alg:positivityGeneral}\
  \begin{algorithmic}[1]
    \REQUIRE{$(G,S)$, where
      \begin{itemize}[leftmargin=*]
      \item $G\subseteq\initial_w(I)$, a reduced Gr\"obner basis of a maximal cone $C_w(I)\subseteq\Trop(I)$,
      \item $S$, a group as in Convention~\ref{con:positivity}.
      \end{itemize}}
    \ENSURE{$S^{+}_w(I)\coloneqq \{\sigma\in S\mid \sigma\cdot C_w(I)\subseteq\Trop^{+}(I)\}$, a set of symmetries which map $C_w(I)$ onto $\Trop^{+}(I)$.}
    \STATE Apply Algorithm~\ref{alg:positivityReduction}:
    \[ H\coloneqq \text{positivityReduction}(G)\subseteq K[x]. \]\vspace{-1em}
    \STATE Apply Algorithm~\ref{alg:positivitySigns}
    \[ R\coloneqq \text{positivitySigns}(H)\subseteq\{\pm 1\}^n. \]\vspace{-1em}
    \STATE Construct
    \[ P\coloneqq \bigcap_{r\in R} \{ \sigma\in S\mid \sigma\cdot r\geq 0\}, \]
    where $S$ acts on $\{\pm 1\}^n$ as it acts on $K^n$.
    \RETURN{$P$}
  \end{algorithmic}
\end{algorithm}

\begin{remark}\label{rem:positivitySigns}
  Computing the signs of a finite set of points $V(J)\subseteq \CC^n$ for a zero-dimensional ideal $J\unlhd\CC[x]$ as in Algorithm~\ref{alg:positivitySigns} can be done symbolically, numerically or with a mix of both.

  One conceptually straightforward option is to approximate $V(J)$ using numerical algebraic geometry. Once a point in $V(J)$ is known with sufficient precision, there are algorithms for certifying reality \cite{HS12} and its sign can simply be read off.

  Alternatively, one can symbolically compute a triangular decomposition of $J$ into factors of the form
  \[ \langle p_1(x_1),x_2^{d_2}-p_2(x_1),\dots,x_n^{d_n}-p_n(x_1)\rangle, \qquad p_i \text{ univariate polynomials}, \]
  from which one can proceed using numerical algorithms for the univariate case.
\end{remark}

\section{The maximal-dimensional cones of $\TropGrass^+(3,8)$}\label{sec:positiveGrassmannian}
\noindent
In this section, we verify \cite[Conjecture 8.1]{SpeyerWilliams05} for the Grassmannian $\Grass_0(3,8)$, which relates the combinatorial structure of the positive tropicalization with the combinatorial structure of a cluster algebra. This serves as a test for the correctness of our computations, as the conjecture has been proven for $\Grass_0(3,8)$ by Brodsky and Stump \cite[Remark 2.23]{BS18}.

Cluster algebras are algebras with a remarkable hidden combinatorial structure. First introduced in \cite{FZ02} by Fomin and Zelevinsky, cluster algebras are subrings of rational function fields $K(x_1,\ldots,x_n)$ generated by a union of overlapping algebraically independent $n$-subsets. These so-called \emph{clusters} are connected through mutations, rules which transform one cluster to another, and together they form a simplicial complex called the \emph{cluster complex}. In \cite{FZ03}, Fomin and Zelevinsky completely classify all cluster algebras of \emph{finite type}, i.e., cluster algebras with finite cluster complexes. Similar to the Cartan-Killing classification of semisimple Lie algebras, their classification associates any finite type cluster algebra a Dynkin graph.
One prominent family of Cluster algebras are Grassmannians $\Grass_0(k,n)$, initially shown by Fomin and Zelevinksy for $k=2$, later fully proven by Scott \cite{Scott06}. 

The conjecture of Speyer and Williams is based on observations on the Grassmannians $\Grass_0(2,n)$, $\Grass_0(3,6)$, and $\Grass_0(3,7)$. By \cite{Scott06}, this makes $\Grass_0(3,8)$ the only remaining Grassmannian whose cluster algebra is of finite type, i.e., whose cluster complex is finite.

\begin{conjecture}[{\cite[Conjecture 8.1]{SpeyerWilliams05}}]\label{conj:SpeyerWilliams}
  Let $\mathcal A$ be a cluster algebra of finite type and $\mathcal S(\mathcal A)$ its associated cluster complex. If the lineality space of $\Trop^+\Spec\mathcal A$ has dimension $|C|$, then $\Trop^+\Spec\mathcal A$ is abstractly isomorphic to the cone over $\mathcal S(\mathcal A)$. If the condition on the lineality space does not hold, the resulting fan is a coarsening of the cone over $\mathcal S(\mathcal A)$.
\end{conjecture}

The conjecture was proven by Brodsky and Stump \cite{BS18} for cluster algebras of type $A$ and of all types of to rank at most $8$, which includes $\Grass_0(3,8)$.

\subsection{Computing the cluster complex $\mathcal S(\Grass_0(3,8))$}
Thanks to an implementation by Stump, \textsc{SAGE} \cite{sagemath} features functions for computing and analyzing cluster complexes. The algorithm is based on a work of Ceballos, Labb\'e, and Stump \cite{CLS14}, and requires the root system of the cluster algebra. The root system for $\Grass_0(3,8)$ is the exceptional group $E_8$ \cite[Theorem 5]{Scott06}:

\begin{lstlisting}[language=Python,
  basicstyle={\ttfamily}]
  C = ClusterComplex(['E', 8]);
\end{lstlisting}

\textsc{Sage} returns an object of type \texttt{cluster complex}, whose maximal cells can be seen via

\begin{lstlisting}[language=Python,
  basicstyle={\ttfamily}]
  C.facets();
\end{lstlisting}

\subsection{Computing the positive tropicalization $\TropGrass^+(3,8)$}\label{sec:TropGrassPlus38}
Using the algorithms in Section~\ref{sec:positiveTropicalization} on the computational results in Section~\ref{sec:grass38}, we obtain:

\begin{proposition}
	There is a Dressian subfan supported on the positive tropical Grassmannian $\TropGrass^+(3,8)$. It is a pure $16$-dimensional subfan of the Dressian $\Dress(3,8)$ in $\RR^{56}$ with an $8$-dimensional lineality space and $f$-vector
	$(120$, $2072$, $14\,088$, $48\,544$, $93\,104$, $100\,852$, $57\,768$, $13\,612)$.
\end{proposition}
\begin{proof}
  By Corollary~\ref{cor:positivityAndSaturation}, positivity only depends on the saturated initial ideals\footnote{Instead of Corollary~\ref{cor:positivityAndSaturation}, we could also rely on the recent results of \cite{ALS20} and \cite{SW20} that the positive tropical Grassmannian $\TropGrass_0^+(k,n)$ equals the positive Dressian $\Dress^+(k,n)$, which implies that the Pl\"ucker structure on $\TropGrass_0^+(k,n)$ is a coarsening of the Gr\"obner structure.}, and, by Theorem~\ref{thm:TropGrass38Saturation}, the saturated initial ideals of $\TropGrass_0(3,8)$ only depend on the Dressian cones. It therefore suffices to check the $4\,766$ $S_8$-orbits of Dressian cones in Theorem~\ref{thm:TropGrass38GroebnerStructure} instead of the $14\,763$ $S_8$-orbits of Gr\"obner cones.

  Of the $4\,766$ saturated initial ideals of $\Grass_0(3,8)$ all but one are binomial and thus admissible for Algorithm~\ref{alg:positivityBinomial}. The unique non-binomial saturated initial ideal arises from the Dressian orbit containing $-e_{015}-e_{024}-e_{067}-e_{126}-e_{137}-e_{235}-e_{346}-e_{457}$, and it has no positive cone in its orbit by Algorithm~\ref{alg:positivityGeneral}. To be more specific, it is not hard to see that the resulting ideal of Algorithm~\ref{alg:positivityReduction} has no real solution at all, for example by eliminating all but one variable.
\end{proof}

Note that \cite{SpeyerWilliams05} considers positive tropicalizations with the coarsest structure refined by the individual Gr\"obner fans of all cluster variables. For the cluster variables of $\Grass_0(3,8)$, recall the following result from \cite{Scott06}:

\begin{theorem}{\cite[Theorem 8]{Scott06}}
  The cluster algebra of $\Grass_0(3,8)$ possesses $128$ cluster variables:
  \begin{description}[leftmargin=*]
  \item[48] Pl\"ucker variables $p_{ijk}$ where $\{i,j,k\}\neq\{i,i+1,i+2\}$ mod $8$.
  \item[56] quadratic Laurent binomials with positive coefficients, inherited from $\Grass_0(3,6)$, describing six points in a special position:
    \[ Y^{123456} = (p_{346})^{-1}\cdot\Big(p_{146}p_{236}p_{345}+ p_{136}p_{234}p_{456}\Big) \quad\text{and}\quad X^{123456} = Y^{612345}\]
    and their $D_8$-translates.
  \item[24] cubic Laurent trinomials with positive coefficients describing eight points in a special position:
    \begin{align*}
      A &= (p_{578})^{-1}\cdot\Big(p_{178}p_{567}\cdot X^{123458}+ p_{158}p_{678}\cdot X^{123457}\Big) \quad \text{and}\\
      B &= (p_{158})^{-1}\cdot\Big(p_{128}p_{567}\cdot X^{123458}+ p_{258}\cdot A \Big)
    \end{align*}
    and their $D_8$-translates.
  \end{description}
\end{theorem}

Since the Gr\"obner fan of the Pl\"ucker variables consist of a single cone that is the whole space, refining with them does not change anything. Hence, it only remains the $80$ polynomials $X$, $Y$, $A$ and $B$.

\begin{theorem}
	The positive tropical Grassmannian $\TropGrass^+(3,8)$ endowed with the Pl\"ucker structure and refined by the Gr\"obner fans of all 120 cluster variables of $\Grass(3,8)$ is a $16$-dimensional pure simplicial fan in $\RR^{56}$ with a $8$-dimensional lineality space and $f$-vector $(128$, $2\,408$, $17\,936$, $67\,488$, $140\,448$, $163\,856$, $100\,320$, $25\,080)$.
  As an abstract simplicial complex, it is isomorphic to the cluster complex $\mathcal S(\Grass(3,8))$.
\end{theorem}
\begin{proof}
  The refinement was straightforwardly computed by intersecting all maximal Dressian cones on $\TropGrass^+(3,8)$ with the maximal cones of the Gr\"obner fans of the cluster variables.

  The isomorphism of the two simplicial complexes was tested using \textsc{Nauty} \cite{nauty} by McKay, which was called in \textsc{Polymake} through the function \texttt{fan::isomorphic}. The function takes two objects of type \texttt{IncidenceMatrix}, in our case:
  \begin{enumerate}
  \item the output of \textsc{SAGE}'s \texttt{C.facets()}, \texttt{C} being the cluster complex of type $E_8$,
  \item the output of \textsc{Polymake}'s \texttt{\$F->MAXIMAL\_CONES}, \texttt{\$F} being the polyhedral fan supported on $\TropGrass^+(3,8)$ described above. \qedhere
  \end{enumerate}
\end{proof}


\section{Open questions}\label{sec:open}

In this section, we briefly discuss some open questions beyond the scope of our article.

\subsection{Coarsest structures on tropical varieties}

A frequently arising question on the geometry of tropical varieties is whether they have a natural coarsest structure, i.e., whether there is a natural coarsest polyhedral complex supported on them. While it is long known that there is no unique coarsest structure \cite[Ex. 5.2]{ST08} and that natural coarsenings of the Gr\"obner fan exist \cite{Cartwright12}, the question remains largely open.

For tropical Grassmannians in characteristic $0$ specifically, the question boils down to the following conjecture which all computations up to and including ours support:
\begin{conjecture}
  For any $w,v\in\TropGrass_0(k,n)$ in the relative interior of a maximal Gr\"obner cone we have
  \[w \text{ and } v \text{ lie on the same cone of } \Dress(k,n) \quad \Longleftrightarrow\quad \initial_w(\mathcal I_{n,k}):p^\infty = \initial_v(\mathcal I_{n,k}):p^\infty, \]
  where $\mathcal I_{n,k}$ is the Pl\"ucker ideal and $(\cdot):p^\infty$ denotes the saturation at the product of all Pl\"ucker variables.
  In particular, there is a subfan of the Dressian $\Dress(n,k)$ which coarsens the Gr\"obner subfan supported on $\TropGrass_0(n,k)$.
\end{conjecture}

In addition to any theoretical insight that such a coarsest structure could offer, the question is of direct relevance for two practical reasons.

On the one hand, it will improve our understanding for the complexity of tropical varieties and consequently also the feasibility of computations in tropical geometry, especially with a view towards applications \cite{LeykinYu19}. Current bounds on the f-vector of general tropical varieties are derived from universal Gr\"obner bases \cite{JoswigSchroeter:2018}, and are thus expected to be far from optimal.

On the other hand, it will help with concrete large scale computations. For $\TropGrass_2(4,8)$, our implementation in Section~\ref{sec:parallelFramework} gets stuck on a handful of isolated Gr\"obner bases containing polynomials of degree $15$, for whom simple division with remainder takes several days. Having a coarser structure might allow us to skip those problematic Gr\"obner cones which are still few and far in between.

\subsection{Positive tropicalizations and cluster complexes}

Section~\ref{sec:positiveTropicalization} contains algorithms for testing whether a maximal Gr\"obner cone $C_w(I)\subseteq\Trop(I)$ is contained in the positive tropicalization $\Trop^+(I)$. It is currently unclear whether this is sufficient to compute $\Trop^+(I)$ as not much is known about its structure. If $I$ is prime, is $\Trop^+(I)$ pure? If $\Trop^+(I)\neq\emptyset$, is $\dim (\Trop^+(I)) = \dim(\Trop(I))$?

In \cite[Section 8]{SpeyerWilliams05} Speyer and Williams moreover suspect an analogue of \cite[Conjecture 8.1]{SpeyerWilliams05} to hold for infinite type cluster algebras. For Grassmannians specifically, this means:
\begin{conjecture}
  Let $\mathcal S_{k,n}$ denote the (possibly infinite) cluster complex of the Grassmannian $\Grass(k,n)$. Then
  \[ \mathcal S_{k,n} = \varprojlim \TropGrass_0^+(k,n), \]
  where the inverse limit is taken over all reembeddings of $\TropGrass_0(k,n)$ into $\RR^{|\Lambda|}$, where $\Lambda$ is any finite set of cluster variables containing the Pl\"ucker variables.
\end{conjecture}

\noindent
It would be interesting to test the conjecture for $\TropGrass_0^+(4,8)$ once it can be computed.

\subsection{Real tropicalizations and the topology of real algebraic varieties}

One motivation for tropical geometry stems from the fact that tropical varieties are capable of faithfully capturing the topology of their complex algebraic counterparts. Under special circumstances, the same holds true for real algebraic varieties, as can been seen in Viro's patchworking \cite{Viro06} and a series of works on polynomial systems with many real solutions \cite{ElHilany18,BSS18,BDIM19}.
Currently, computing the topology of real algebraic varieties remains by and large a challenging open problem with promising solutions only for special cases \cite{BHR18}.

Our algorithms in Section~\ref{sec:positiveTropicalization} can easily be modified to test maximal cones for inclusion in the following notion of real tropicalization:

\begin{definition}~\label{def:realTropicalization}
  We define the \emph{real tropicalization} of an ideal $I\unlhd K[x]$ to be
  \[ \Trop_R(I) \coloneqq \text{cl}\Big(\nu\big(V(I)\cap (R^\ast)^n\big)\Big)\subseteq\RR^n, \]
  where $R$ denotes the subset of complex Puiseux series whose lowest coefficient is real and $\text{cl}(\cdot)$ the closure in the euclidean topology.
\end{definition}

In the light of recent works on the topic of real tropicalizations \cite{JSY18}, albeit slightly different to Definition~\ref{def:realTropicalization}, a natural question is whether the real tropicalization as in Definition~\ref{def:realTropicalization} be used to compute the topology of real algebraic varieties in practice.


\renewcommand*{\bibfont}{\small}
\printbibliography

\end{document}